\documentclass{amsart}
\usepackage{graphicx} 
\usepackage{amsmath, amssymb, mathrsfs, verbatim, multirow}
\usepackage[all]{xy}
\usepackage{hyperref}
\usepackage{caption}
\usepackage{imakeidx}
\usepackage{tikz}
\usepackage{tkz-tab}
\usepackage{subcaption} 
\usepackage{float}

\newtheorem{Teo}{Theorem}[section]
\newtheorem{Def}[Teo]{Definition}
\newtheorem{Prop}[Teo]{Proposition}

\newtheorem{Lema}[Teo]{Lemma}
\newtheorem{Lem}[Teo]{Lemma}
\newtheorem{LDef}[Teo]{Lemma-Definition}

\newtheorem{Cor}[Teo]{Corollary}

\theoremstyle{definition}
\newtheorem{Obs}[Teo]{Remark}

\newcommand{\N}{\mathbb{N}}

\newcommand{\Llr}{\ \Longleftrightarrow\ }
\newcommand{\lra}{\longrightarrow}

\def\vb{\bar{v}}

\def\al{\alpha}

\def\bad{B(a,\dta)}
\def\bb{\mathcal{B}}
\def\bopt{\bb^{\op{opt}}}

\def\cc{\mathcal{C}}

\def\dd{\mathcal D}

\def\dta{\delta}

\def\e{\medskip}
\def\ep{\epsilon}
\def\epb{\overline{\epsilon}}
\def\epm{\epsilon_\mu}
\def\epn{\epsilon_\nu}
\def\epm{\epsilon_\mu}

\def\g{\Gamma}
\def\ga{\gamma}
\def\gge{\mathcal{G}_\eta}
\def\ggm{\mathcal{G}_\mu}

\def\ggn{\mathcal{G}_\nu}
\def\ggo{\mathcal{G}_{\overline \mu}}
\def\gm{\g_\mu}
\def\hk{\hookrightarrow}

\def\imp{\quad\Longrightarrow\quad}
\def\ino{\op{in}_{\overline \mu}}
\def\inu{\op{in}_\nu}
\def\inm{\op{in}_\mu}
\def\inn{\op{in}}
\def\irr{\op{Irr}}

\def\kb{\overline{K}}
\def\kbx{\kb[x]}
\def\kh{K^h}
\def\khx{K^h[x]}
\def\km{k_\mu}
\def\kp{\operatorname{KP}}
\def\kpi{\operatorname{KP}_\infty}
\def\kpm{\operatorname{KP}(\mu)}
\def\kx{K[x]}
\def\ks{K^{\op{sep}}}

\def\La{\Lambda}

\def\mink{\operatorname{Min}_K}

\def\om{\omega}
\def\omu{\overline\mu}
\def\op{\operatorname}
\def\opt{\operatorname{Opt}}

\def\ps{\partial_s}
\def\res{\operatorname{res}}
\def\resk{\operatorname{res}_K}

\def\sg{\sigma}
\def\sii{\quad\Longleftrightarrow\quad}
\def\smu{\sim_\mu}

\def\ss{\mathcal S}
\def\sub{\subseteq}

\def\tb{\overline{\vv}}
\def\tme{\ty(\mu,\eta)}
\def\tmn{\ty(\mu,\nu)}
\def\ttt{\mathcal{T}}
\def\ty{\mathbf{t}}

\def\vad{v_{a,\delta}}
\def\vb{\bar{v}}
\def\vh{v^h}
\def\vv{\mathcal{V}}

\title{A topological approach to key polynomials}
\author{Enric Nart}
\address{Departament de Matem\`{a}tiques,         Universitat Aut\`{o}noma de Barcelona,         Edifici C, E-08193 Bellaterra, Barcelona, Catalonia}
\email{enric.nart@uab.cat}
\author{Josnei Novacoski}
\address{Departamento de Matem\'{a}tica,         Universidade Federal de S\~ao Carlos, Rod. Washington Luís, 235, 13565--905, S\~ao Carlos -SP, Brazil}
\email{josnei@ufscar.br}

\author{Giulio Peruginelli}
\address{Dipartimento di Matematica  ``Tullio Levi-Civita", Universit\`a Degli Studi di Padova,
Via Trieste 63, I-35121 Padova, Italy}
\email{gperugin@math.unipd.it}

\thanks{Partially supported by grant PID2020-116542GB-I00  funded by the Spanish MCIN/AEI. During the realization of this project the second author was supported by a grant from Funda\c{c}\~ao de Amparo \`a Pesquisa do Estado de S\~ao Paulo (process numbers 2017/17835-9) and a grant from Conselho Nacional de Desenvolvimento Cient\'ifico e Tecnol\'ogico (process number 303215/2022-4).}
\keywords{ultrametric balls, key polynomials, abstract key polynomials, limit key polynomials}
\subjclass[2020]{Primary 13A18}

\begin{document}

\begin{abstract}


In this paper we present characterizations of the sets of key polynomials and abstract key polynomials for a valuation $\mu$ of $K(x)$, in terms of (ultrametric) balls in the algebraic closure $\overline K$ of $K$ with respect to $v$, a fixed extension of $\mu_{\mid K}$ to $\overline K$.  In particular, we show that the ways of augmenting $\mu$, in the sense of Mac Lane, are in one-to-one correspondence with the partition of a fixed closed ball $B(a,\delta)$ associated to $\mu$ into the disjoint union of open balls $B^\circ(a_i,\delta)$, modulo the action of the decomposition group of $v$. We also present a similar characterization for the set of limit key polynomials for  an increasing family of valuations of $K(x)$.
\end{abstract}

\maketitle

\section{Introduction}
For a valued field $(K,v)$ we consider a fixed extension of $v$ (which we denote again by $v$) to the algebraic closure $\overline K$ of $K$. Set $\g=v\overline K$, which is the divisible hull of $vK$, and consider a fixed embedding  $\g\hk\La$ into a divisible ordered abelian group $\La$. For $a\in \overline K$ and $\delta\in \La$, we define the closed and open balls with center on $a$ and radius $\delta$ by
\[
B(a,\delta)=\{b\in\overline K\mid v(b-a)\geq\delta\}\,\supseteq\,B^\circ(a,\delta)=\{b\in \overline K\mid v(b-a)>\delta\},
\]
respectively. We consider the set of such balls in $\overline K$:
\[
\mathbb{B}:=\{B(a,\delta) \mid (a,\delta)\in \overline{K}\times\La\},\;\;
\mathbb{B}^\circ:=\{B^\circ(a,\delta) \mid (a,\delta)\in \overline{K}\times\La\}.
\]
Let $\vv=\vv(v,\La)$ be the set of all extensions of $v$ to $K(x)$, taking values in $\La$. 
For every $\mu\in\vv$ we denote by ${\rm KP}(\mu)$ the set of \textit{key polynomials} for $\mu$ (Definition \ref{keypolyno}), and we let $\Psi(\mu)$ be the set of  \textit{abstract key polynomials} for $\mu$ (Definition \ref{absKP}). Finally, ${\rm KP}_\infty(\mathcal C)$ will be the set of \textit{limit key polynomials} for an increasing family $\mathcal C$ of valuations in $\vv$ (Definition \ref{Limitkp}).

The main goal of this paper is to describe the sets ${\rm KP}(\mu)$, $\Psi(\mu)$ and ${\rm KP}_\infty(\mathcal C)$ in terms of elements of $\mathbb{B}$ and $\mathbb{B}^\circ$.


The motivation for this work comes from the following reasoning. If $K$ is algebraically closed, then both sets ${\rm KP}(\mu)$ and $\Psi(\mu)$ are subsets of $\{x-a\mid a\in K\}$ (because every polynomial in these sets is irreducible and monic). Also, in this case for every valuation-transcendental valuation $\mu\in\vv$ there exists $(a,\delta)\in K\times \La$ such that $\mu=v_{a,\delta}$ where
\begin{equation}\label{monomialvalu}
v_{a,\delta}\left(a_0+a_1(x-a)+\ldots+a_r(x-a)^r\right):=\min_{0\leq k\leq r}\{v(a_k)+k\delta\}. 
\end{equation}
In particular,  for algebraically closed valued fields, the situation is much easier to handle. This was the leitmotif of some classical papers (for instance, \cite{AlePop, APZTheorem, APZ, APZ2, Kuhl, PP}), where some properties of valuations on a polynomial ring $\kx$ were deduced from an analysis of their extensions to $\kbx$.


This paper provides a precise and more complete picture of this approach, in what concerns the description of the sets ${\rm KP}(\mu)$, $\Psi(\mu)$ and ${\rm KP}_\infty(\mathcal C)$.

The notion of key polynomial was first introduced by Mac Lane in \cite{mcla} in order to describe all the extensions of a rank one discrete valuation $v$ on a field $K$ to the field of rational functions $K(x)$. The main idea of these objects is the following. For a given valuation $\mu\in \vv$ there might exist $\phi\in K[x]$ and $\eta\in\vv$ such that
\begin{equation}\label{eqaugornf}
\mu(\phi)<\eta(\phi)\mbox{ and }\mu(f)\leq \eta(f),\mbox{ for every }f\in K[x].
\end{equation}
If that happens we write $\mu<\eta$. Key polynomials for $\mu$ are polynomials $\phi$ of minimal degree  satisfying property \eqref{eqaugornf} (for a fixed $\eta$). 
Such polynomials allow us to construct \textit{ordinary augmentations} of $\mu$ (roughly speaking, these augmentations are valuations $\eta$ that are ``minimal" among valuations satisfying \eqref{eqaugornf}). By doing this, one can construct a ``tree" of valuations as follows: let $\mathcal V_0$ be the set of all valuations on $K(x)$ which are of the form \eqref{monomialvalu}. Each valuation $\mu_0\in \mathcal V_0$ can be augmented by using elements of ${\rm KP}(\mu_0)$. Denote by $\mathcal V_1$ the set of all valuations obtained on this way. We can iterate this process and construct similar sets $\mathcal V_n$ for every $n\in\N$.

We can also consider \emph{stable limit} valuations. 
Name\-ly, consider an infinite family of ordinary augmentations
\[
\mu_0\lra \mu_1\lra \ldots\lra \mu_n\lra
\]
and suppose that for every $f\in K[x]$ there exists $n_f\in\N$ such that
\[
\mu_n(f)=\mu_{n_f}(f)\mbox{ for every }n\geq n_f.
\] 
Then the map $\mu(f):=\mu_{n_f}(f)$ defines a valuation on $K(x)$. Denote the set of all these valuations by $\mathcal V_\infty$.
The main goal of Mac Lane's work was to show that:
\[
\mathcal V=\left(\bigcup_{n\in \N_0} \mathcal V_n\right)\cup \mathcal V_\infty,
\]
if $vK$ is discrete of rank one  \cite[Theorem 8.1]{mcla}. 
 
In order to prove a similar result when $vK$ is not discrete, or has larger rank, Vaquié introduced the notion of \textit{limit key polynomial}. With this notion, one can construct \textit{limit augmentations}, which are the equivalent of ordinary augmentations for increasing families of valuations that do not have a stable limit. Vaquié showed in \cite{Vaq} that if one combines both constructions, then one can obtain the whole set $\mathcal V$. 

Abstract key polynomials were introduced in \cite{Dec} and \cite{NS2018} (in \cite{NS2018} they are called key polynomials). They are related to the ideas presented in \cite{hos} and \cite{hos1}. These objects are closely related to those presented by Mac Lane and Vaquié. The main idea is that, instead of building a set of valuations (for instance, the sets $\mathcal V_n$, $n\in \N_0\cup\{\infty\}$ as before) and checking whether all the valuations in $\mathcal V$ are of such form, abstract key polynomials allow us to determine the relevant polynomials in the construction of a fixed element in $\mathcal V$.

The above discussion illustrates the difference between the sets ${\rm KP}(\mu)$ and $\Psi(\mu)$. While the polynomials in ${\rm KP}(\mu)$ allow us to find another valuation $\eta$ for which $\mu\leq\eta$, the elements in $\Psi(\mu)$ allow us to determine valuations $\nu$ such that $\nu\leq \mu$.

For a given valuation-transcendental $\mu\in \vv$ we can fix an extension $\overline\mu$ of $\mu$ to $\overline K(x)$. Then $\overline \mu$ is of the form \eqref{monomialvalu} for some $a\in\overline K$, $\delta\in\La$. The idea of using elements of $\mathbb{B}$ to describe valuations comes from the fact that if $(a,\delta),(b,\epsilon)\in \overline{K}\times \La$, then
\[
v_{a,\delta}(f)\leq v_{b,\epsilon}(f)\mbox{ for all }f\in \overline K[x] \Llr \dta\le\ep\ \mbox{ and }\ B(a,\delta)\supseteq B(b,\epsilon)\ .
\]
Suppose that $\mu=(v_{a,\delta})_{\mid K[x]}$ for some $a\in\overline K$,  $\delta\in\La$. The valuations $\eta$ obtained from key polynomials for $\mu$ have the property $\mu\leq\eta$. Hence, it is natural to try to obtain ${\rm KP}(\mu)$ from elements in $\mathbb B$ or $\mathbb B^\circ$ contained in $B(a,\delta)$.
On the other hand, the valuations $\nu$ obtained from abstract key polynomials for $\mu$ satisfy $\nu\leq \mu$. Hence,  it is natural to try to obtain $\Psi(\mu)$ from elements in $\mathbb B$ or $\mathbb B^\circ$ containing $B(a,\delta)$.

Take any such ball $B$ (containing or contained in $B(a,\delta)$). A natural candidate to be key polynomial or abstract key polynomial is  the minimal polynomial of $b$ over $K$ for some $b\in B$. We obtain the right object by considering elements $b\in B$ whose degree over $K$ is minimal.

The first main result of this paper (Theorem \ref{abskeypol1}) tells us that $\Psi(\mu)$ is exactly the set of all minimal polynomials of elements   of minimal degree over $K$, belonging to balls $B\in\mathbb B$ such that $B\supseteq B(a,\delta)$, where $(\vad)_{\mid \kx}\le\mu$. In Section \ref{subsecCompleteseq}, we improve this result to obtain \textit{complete sets} of abstract key polynomials for a given $\mu$, that do not have ``redundant" elements.


The second set of results is presented in Section \ref{secKP} and deals with key polynomials for $\mu=(v_{a,\delta})_{\mid K[x]}$. The main result (Theorem \ref{mainthm}) says that for $c\in B(a,\delta)$, the minimal polynomial of $c$ belongs to ${\rm KP}(\mu)$ if and only if $c$ has minimal degree over $K$ among all the elements in $B^\circ(c,\delta)$. An alternative way of seeing this is the following. Let
\[
B(a,\delta)=\bigsqcup_{i\in I}B^\circ(a_i,\delta)  
\]
be the partition of $B(a,\delta)$ as a disjoint union of open balls of radius $\delta$. For each $i\in I$, let $\kp_i(\mu)$ be the set of minimal polynomials of the elements in $B^\circ(a_i,\delta)$ of minimal degree over $K$. Then,
\[
{\rm KP}(\mu)=\bigcup_{i\in I}{\rm KP}_i(\mu).
\]
We also present a criterion (Lemma \ref{lemamnartomp}) to determine when ${\rm KP}_i(\mu)={\rm KP}_j(\mu)$  for different $i,j\in I$. One important tool used to prove the results in Section \ref{secKP} is Theorem \ref{Nartgoingup}, which can be seen as a ``going-up" and ``going-down" property for valuations on $K(x)$ (and their extensions to $\overline{K}(x)$).

We also deal with limit key polynomials for a family $\mathcal C=\{\mu_i\}_{i\in I}$ of increasing valuations in $\vv$. The first relevant result (Theorem \ref{teoextanuchainvalu}) is that for each such family, there exists another increasing family $\overline{\mathcal C}=\{\overline \mu_i\}_{i\in I}$ of valuations on $\overline K(x)$ such that $(\overline\mu_i)_{\mid K[x]}=\mu_i$ for every $i\in I$. In particular, if each $\overline \mu_i$ is of the form $v_{a_i,\delta_i}$ for some $a_i\in \overline K$, $\delta_i\in \La$, then the balls $\left\{B(a_i,\dta_i)\right\}_{i\in I}$ form a descending chain, and we can easily deduce (Theorem \ref{mainlimitkp}) that ${\rm KP}_\infty(\mathcal C)$ is the set of minimal polynomials over $K$ of elements of minimal degree over $K$ belonging  to  the ball

\[
B:=\bigcap_{i\in I}B(a_i,\delta_i).
\]

Finally, we use the results in this paper to obtain a \textit{Mac Lane--Vaquié chain} for a valuation in $\vv$ (Theorem \ref{mainmlvchain}). This is done by using the construction of $\mu$-\textit{optimal} ultrametric balls appearing in Section \ref{subsecCompleteseq}.  

We emphasise that this paper was motivated by the results of \cite{Andrei}, \cite{NartJosnei}, \cite{N2019}  and specially \cite{Vaq3}. In particular, some of the results in this paper already appear (at least in particular cases), or can be deduced from the results appearing there. 

\section{Preliminaries}\label{secPrelim}

Let $\vv=\vv(v,\La)$ be the tree of all $\La$-valued extensions of $v$ to the field $K(x)$. We identify every $\mu\in\vv$ with a valuation on the polynomial ring,
\[
 \mu\colon \kx\lra \La\cup\{\infty\},
\]
with trivial support; that is, $\mu^{-1}(\infty)=\{0\}$. The residue field of $(K,v)$ will be denoted by $K\!v$. The value group and residue field of $\mu$ will be denoted by $\gm$ and $\km$, respectively. 

This set $\vv$ has a partial ordering. For $\mu,\eta\in\vv$, we define
\[
\mu\leq \eta\ \sii\  \mu(f)\leq \eta(f)\quad \mbox{for all }f\in K[x].
\]
We say that $\vv$ is a \textbf{tree} because the intervals 
\[
(-\infty,\mu]=\left\{\nu\in \vv\mid \nu\le\mu\right\}
\]
are totally ordered for all $\mu\in\vv$ \cite{VT}.

Any node $\mu\in\vv$ is (exclusively) of one of the following types:
\begin{itemize}
	\item\textbf{Value-transcendental}: $\gm/vK$ is not a torsion group.
	\item\textbf{Residue-transcendental}: the extension $\km/K\!v$  is transcendental.
	\item\textbf{Valuation-algebraic}:  $\gm/vK$ is a torsion group and $\km/K\!v$  is algebraic.
\end{itemize}

We say that $\mu$ is \textbf{valuation-transcendental} if it is value-transcen\-dental or residue-transcendental. These are precisely the inner (non-maximal) nodes of $\vv$.

\subsection{Key polynomials}
The {\bf graded algebra} of $\mu\in\vv$ is the integral domain $\ggm=\bigoplus_{\alpha \in \Gamma_\mu}\mathcal P_\alpha/\mathcal P_\alpha^+$, where
\[
\mathcal P_\alpha=\{f\in K[x]\mid \mu(f)\geq \al\}\supseteq\mathcal P_\alpha^+=\{f\in K[x]\mid \mu(f)> \al\}.
\]

Every $f\in K[x]\setminus\{0\}$ has a homogeneous \textbf{initial coefficient}  $\inm f\in\ggm$, defined as the image of $f$ in $\mathcal P_{\mu(f)}/\mathcal P_{\mu(f)}^+\sub \ggm$.

\begin{Def}\label{keypolyno}
A monic $\phi\in K[x]$ is a \textbf{key polynomial}  for $\mu$ if  $(\inm\phi)\ggm$ is a homogeneous prime ideal containing no initial coefficient \ $\inm f$ with  $\deg f< \deg \phi$.
\end{Def}

 We denote by $\kpm$ the set of all key polynomials for $\mu$. These polynomials are  necessarily irreducible in $\kx$.
Let  us recall a criterion  for the existence of key polynomials \cite[Theorem 4.4]{KP}, \cite[Theorem 2.3]{MLV}.

\begin{Teo}\label{EKP}
	For every $\mu\in\vv$ the following conditions are equivalent.
	\begin{enumerate}
		\item[(a)] $\kpm=\emptyset$.
		\item[(b)] $\ggm$ is a simple algebra (all nonzero homogeneous elements are units).
		\item[(c)] $\mu$  is valuation-algebraic.
		\item[(d)] $\mu$  is a leaf (maximal node) of $\vv$.
	\end{enumerate}
\end{Teo}

\noindent{\bf Definition. }{\it 
	Let $\mu\in\vv$ be valuation-transcendental. We define the \textbf{degree} of $\mu$ as the smallest degree of the key polynomials for $\mu$. We denote it by $\deg(\mu)$. }\e

The degree function preserves the ordering in $\vv$. If $\mu<\nu$ are both valuation-transcendental, then $\deg(\mu)\le\deg(\nu)$ \cite[Lemma 2.2]{VT}.

On the set $\kpm$ we consider the following equivalence relation: 
\[
f\smu g\sii \inm f=\inm g\sii \mu(f-g)>\mu(f).
\]

\begin{Def}
Let $\mu,\eta\in\vv$, with $\mu<\eta$. We define the \textbf{tangent direction} of $\mu$ determined by $\eta$  as the set $\mathbf t(\mu,\eta)$ of all  monic polynomials $\phi\in K[x]$ of smallest degree satisfying  $\mu(\phi)<\eta(\phi)$.
\end{Def}


The following basic properties of the set $\tme$ are proved in \cite[Theorem 1.15]{Vaq} (see also \cite[Lemma 5.3]{PP}) and \cite[Corollaries 2.5 + 2.6]{MLV}.

\begin{Lem}\label{basicTd}
 Let $\mu$, $\eta$  be valuations in $\vv$ such that $\mu<\eta$. Let  $\phi \in\tme$. 
 \begin{enumerate}
 \item [(i)] The polynomial $\phi$ belongs to $\kpm$ and $\tme=\{\chi\in\kpm\mid \phi\smu \chi\}$. 
 \item [(ii)]     For all nonzero $f\in\kx$, we have 
 \begin{itemize}
\item $\mu(f)<\eta(f) \sii \inm \phi\mid \inm f\ \mbox{ in  }\ \ggm$.  
 \item $\mu(f)=\eta(f)\;\imp \inn_{\eta}f \mbox{ is a unit in }\ \gge$.
 \end{itemize}
\end{enumerate}
\end{Lem}

 A tangent direction is a kind of  ``germ of augmentations" of $\mu$. The following result gives a precise meaning  for this terminology.

\begin{Lem}\label{disjointTd}
 Let $\mu,\eta,\nu \in \vv$ such that $\mu<\eta$ and $\mu<\nu$. 
 \begin{enumerate}
 \item [(i)] $\,\tme=\tmn$\, if and only if $\,(\mu,\eta]\cap(\mu,\nu]\ne\emptyset$. 
 \item [(ii)] If $\,\tme\cap\tmn\ne\emptyset$, then $\,\tme=\tmn$.
\end{enumerate}
\end{Lem}
 
\begin{proof}
Item (i) was proved in \cite[Proposition 2.4]{VT}.
Item (ii) is an immediate consequence of Lemma \ref{basicTd} (i).
\end{proof}

\subsection{Abstract key polynomials}\label{secEp}

For all $s\in \N$, the $s$-th Hasse-Schmidt derivative $\partial_s$ on $\kx$ is defined by:
$$
f(x+y)=\sum\nolimits_{0\le s}(\ps f)y^s \quad\mbox{for all } \,f\in\kx,
$$
where $y$ is another indeterminate.

Take $\mu\in\vv$. If $f\in\kx$ is non-constant, we define
$$
\epm(f)=\max\left\{\dfrac{\mu(f)-\mu(\ps{f})}s\ \Big|\ s\in\N\right\}\in\La.
$$

Let us recall \cite[Prop. 3.1]{N2019}, which clarifies the meaning of this value.

\begin{Teo}\label{nov}
	Let $\overline\mu$ be an arbitrary extension  to $\kbx$ of some $\mu\in\vv$. Then, for every non-constant $f\in \kx$ we have 
	$$
	\epm(f)=\max\{\overline\mu(x-a)\mid a\in \op{Z}(f)\},
	$$
	where $\op{Z}(f)$ is the set of roots of $f$ in $\kb$. 
\end{Teo}

We say that $a\in \op{Z}(f)$  is an \textbf{optimizing root}  of $f$ if it satisfies  $\overline\mu(x-a)=\epm(f)$ for some extension $\omu$ of  $\mu$ to $\kbx$.

\begin{Def} \label{absKP}
A monic $Q\in\kx$ is said to be an \textbf{abstract key polynomial} for $\mu$ if for every non-constant $f\in\kx$ we have
\[
 \deg f<\deg Q\imp \epm(f)<\epm(Q).
\]
We denote by $\Psi(\mu)$ the set of all abstract key polynomials for $\mu$.
\end{Def}

For every polynomial $Q\in\kx$ consider the truncated function $\mu_Q$ on $\kx$, defined as follows  on $Q$-expansions:
\[
f=\sum\nolimits_{i\ge0}f_i\, Q^i,\ \deg f_i<\deg Q\ \imp\ \mu_Q(f)=\min\{\mu(f_i\,  Q^i)\mid i\ge0\}.
\]

For every abstract key polynomial $Q$, the truncated function $\mu_Q$ is a valuation such that $\mu_Q\le\mu$ \cite[Proposition 2.6]{NS2018}.

\begin{Lema}\label{maximal}\cite[Corollary 2.22]{AFFGNR} Let $\mu\in\vv$ and let $Q\in\kx$ be a monic polynomial. The following conditions are equivalent.
\begin{enumerate}
	\item[(i)] \ $Q\in\Psi(\mu)$ and $\mu_Q=\mu$.
	\item[(ii)] \ $Q\in\kpm$ and has minimal degree in this set.
\end{enumerate}\end{Lema}


\subsection{Minimal pairs}\label{subsecMP}
Take $(a,\dta)\in \kb\times \La$. The function $\vad$ defined in (\ref{monomialvalu}) is a valuation on $\overline K[x]$,  called \textbf{monomial}, admitting $x-a$ as a key polynomial. For any other pair $(b,\ep)\in \kb\times \La$ we have
\[
\vad=v_{b,\ep}\sii \dta=\ep \ \mbox{ and }\ b\in \bad.
\]
If $\dta$ and $\ep$ belong to $\g$, then this condition is equivalent to $\bad=B(b,\ep)$.


It is well known that every valuation-transcendental $\mu\in \vv$ is the restriction of $\vad$ to $K[x]$ for some $(a,\dta)\in \kb\times\La$
(see for example \cite{AlePop,APZTheorem, APZ2}). Moreover, we have 
\begin{displaymath}
	\begin{array}{lcl}
		\mu\mbox{ is residue-transcendental}&\Llr &\delta\in \Gamma,\\
		\mu\mbox{ is value-transcendental}&\Llr &\delta\in \La \setminus\Gamma.
	\end{array}
\end{displaymath}

For every $c\in \kb$, let $p_K(c)\in\kx$ be the minimal polynomial of $c$ over $K$. 
For every subset $S\subseteq \overline K$, we denote
\[
p_K(S)=\{p_K(c)\mid c\in S\}\subseteq K[x].
\]
Also, consider the following notation
\begin{displaymath}
\begin{array}{rcl}
\deg_K c&:=&\deg\left(p_K(c)\right)\\[5pt]
\deg_KS&:=&\min\{\deg_K c\mid c\in S\}\\[5pt]
\mink S&:=&\{c\in S\mid \deg_K c=\deg_KS\}
\end{array}
\end{displaymath}

The following definition comes from \cite{APZ}.

\begin{Def} 
 The pair $(a,\delta)\in\kb\times\La$ is called a \textbf{minimal pair} over $K$ if $a\in \mink B(a,\delta)$. In this case, we say that $(a,\delta)$ is a \textbf{minimal pair of definition} of the valuation $(\vad)_{\mid \kx}$.
\end{Def}

Let us recall a characterization of abstract key polynomials in terms of minimal pairs which follows from the combination of  \cite[Theorem 1.1]{N2019} and Lemma \ref{maximal}.

\begin{Teo}\label{mp}
Let $\mu\in\vv$. Take a monic, irreducible  $Q\in\kx$ and let $a\in\kb$ be an optimizing root of $Q$. Let $\dta=\epm(Q)$. Then,
\begin{enumerate}
\item[(i)\,] $Q\in\Psi(\mu)$ if and only if $(a,\dta)$ is a minimal pair.
\item[(ii)] $Q\in\kpm$ and has minimal degree in this set if and only if $(a,\dta)$ is a minimal pair of definition of $\mu$.	
\end{enumerate}
\end{Teo}

\subsection{Limit key polynomials}\label{subsecLKP}

Consider an increasing family $\mathcal C=\{\mu_i\}_{i\in I}\subset \vv$. This means that $I$ is a totally ordered set and for every $i,j\in I$, $i<j$, we have $\mu_i<\mu_j$. A polynomial $f$ is said to be \textbf{$\mathcal C$-stable} if there exists $i_0\in I$ such that
\[
\mu_i(f)=\mu_{i_0}(f)\ \mbox{ for every }i\in I,\ i\geq i_0.
\]
Otherwise, it is called \textbf{$\mathcal C$-unstable}, in which case, we have
\[
\mu_i(f)<\mu_j(f)\ \mbox{ for every }i,j\in I\mbox{ with }i<j.
\]

\begin{Def}\label{Limitkp}
	A monic polynomial $\phi$ is called a \textbf{limit key polynomial} for $\mathcal C$ if it is $\mathcal C$-unstable and has the smallest degree among $\mathcal C$-unstable polynomials.
	
	We denote by $\kpi(\cc)$ the set of all limit key polynomials for $\cc$.
\end{Def}

If all polynomials are $\cc$-stable, then $\kpi(\cc)=\emptyset$ and $\cc$ determines a limit valuation  $\mu=\lim_{i\in I}\mu_i=\lim(\cc)$ as follows:
\begin{equation}\label{limit C}
\mu(f):=\max\{\mu_i(f)\mid i\in I\}.
\end{equation}
In this case, we say that $\mu$ is the \textbf{stable limit} of $\cc$, and we define
\[
\deg(\mu):=\max\{\deg(\mu_i)\mid i\in I\}.  
\]
We agree that $\deg(\mu)=\infty$ if the right-hand set is unbounded. 

A valuation $\mu\in\vv$ is the stable limit of some increasing family if and only if it is valuation-algebraic \cite[Proposition 4.1]{VT}.\e



\section{Lifting chains of valuations}
Consider the tree $\tb$  of all $\La$-valued valuations on the field $\kb(x)$, whose restriction to $\kb$ is $v$.
We have a natural restriction map:
\[
\resk\colon\tb\lra\vv,\qquad \omu\longmapsto \resk(\omu):=\omu_{\mid\kx}.
\]
Whenever we say that a valuation $\overline \mu$ on $\kb(x)$ is an \textbf{extension}  of some $\mu\in\vv$, we mean that $\overline\mu\in\tb$ and $\resk (\overline\mu)=\mu$; that is, $\overline\mu$ is a common extension of $\mu$ and $v$.

The main goal of this section is to prove the following result.

\begin{Teo}\label{Nartgoingup}
Let $\nu<\mu\,\le\, \eta$  be valuations on $\kx$ and let $\overline{\mu}$  be an extension of $\mu$ to $\kbx$. Then, there exist extensions $\overline \nu$ of $\nu$, and $\overline \eta$ of $\eta$, to $\kbx$ such that
\[
\overline\nu<\overline\mu\,\le\, \overline\eta.
\]
\end{Teo}

Theorem \ref{Nartgoingup} can be presented in two parts: ``from large to small" and ``from small to large". The ``from large to small" part is well-known and we present a short proof for sake of completeness.

Consider the image set $E_\nu=\{\epn(f)\mid f\in \kx\}$ of the function $\epn$. The existence of $\max(E_\nu)$ characterizes the valuation-transcendental valuations.

\begin{Teo}\label{max}\cite[Theorem 3.4]{Vaq3}
 A valuation $\nu\in\ttt$ is valuation-transcendental if and only if the set $E_\nu$ contains a last element.   
\end{Teo}

Finally, let us mention which polynomials achieve a maximal $\epn$-value.   

\begin{Prop}\label{emax}\cite[Lemma 2.12]{Rig}
For arbitrary $\nu\in\ttt$ and nonconstant $f\in\kx$, we have $\epn(f)=\max(E_\nu)$ if and only if  $\inu f$ is not a unit in the graded algebra $\ggn$.  
\end{Prop}

\begin{Cor}\label{preserve}
Let $\La\sub\La^c$ be a completion of $\La$ with respect to the order topology. Then, $\delta(\nu)=\sup(E_\nu) \in \La^c$ determines a function \,$\delta\colon \ttt\lra \La^c\cup\{\infty\}$ which strictly preserves the ordering. 	
\end{Cor}

\begin{proof}
Suppose $\mu<\nu$ in $\ttt$. 
Let $\ty(\mu,\nu)=[\phi]_\mu$ for some $\phi\in\kpm$. It is easy to check that $\epm(\phi)<\epn(\phi)$ \cite[Lemma 2.7]{Rig}. Also, since $\inm \phi$ is a prime element in $\ggm$, Proposition \ref{emax} shows that $\epm(\phi)=\epb(\mu)$. Thus, $\delta(\mu)=\epm(\phi)<\epn(\phi)\le\delta(\nu)$. 
\end{proof}\e

The following classical result relates some properties of valuations on $\kx$ and their extensions to $\kbx$ \cite{N2019,Rig}.
\begin{Prop}\label{down}
For some $a\in\kb$, $\dta\in\La\cup\{\infty\}$, let $f=\irr_K(a)$ and  $\nu=\resk(\om_{a,\dta})$.
\begin{enumerate}
\item[(i)] \ $\epn(f)=\delta(\nu)=\dta$.
\item[(ii)] \ If $g\in\kx$ satisfies $\epn(g)<\dta$, then $\nu(g)=\vb(g(a))$. 
\end{enumerate}
\end{Prop}

\begin{Lema}\label{large}
Suppose that $\nu<\mu$ in $\vv$ and $\omu\in\tb$ is an extension of $\mu$. Then, there exists an extension $\overline{\nu}\in\tb$ of $\nu$ such that $\overline\nu<\omu$.	
\end{Lema}
\begin{proof}
Suppose  $\mu$  valuation-transcendental and let $\overline \mu=v_{a,\dta}$. By Proposition \ref{down}, for $f=\irr_K(a)$ we have $\dta=\epsilon_\mu(f)=\delta(\mu)$. Take $\overline \nu=v_{a,\ep}$ for $\ep=\delta(\nu)$. Clearly, $\overline \nu<\omu$, because $\ep<\dta$. By Proposition \ref{down}, $\delta(\resk(\overline{\nu}))=\delta(\nu)$. Since $\ttt$ is a tree, from $\nu,\resk(\overline\nu)<\mu$ we deduce that $\nu\le \resk(\overline\nu)$ or $\nu\ge \resk(\overline \nu)$. Corollary \ref{preserve} shows that $\nu=\resk(\overline \nu)$.

Now, suppose $\mu$ valuation-algebraic. Then, $\overline\mu$ is the stable limit of a family $(\overline \mu_i)_{i\in A}$ of valuation-transcendental valuations in $\tb$ whose corresponding ultrametric balls have empty intersection \cite[Thm. 3.23]{Vaq3}. In particular, $\overline\mu_i<\overline \mu$ for all $i\in A$.

It is easy to check that $\mu$ is then the stable limit of the family $\mu_i=\resk(\overline\mu_i)$. Thus, there exists some $i\in A$ such that $\nu<\mu_i<\mu$. As shown above, there exists $\overline\nu\in\tb$ such that $\resk(\overline\nu)=\nu$ and $\overline\nu<\overline\mu_i$. This ends the proof.      
\end{proof}\e

Thus, the real aim of this Section is to prove the ``from small to large" part.

\begin{Teo}\label{small}
Suppose that $\mu<\eta$ in $\vv$ and $\overline\mu\in\tb$ is an extension of $\mu$. Then, there exists an extension $\overline\eta\in\tb$ of $\eta$ such that $\overline\mu<\overline\eta$.	
\end{Teo}

This result was proved by Vaqui\'e in the particular situation in which $(K,v)$ is Henselian and $\eta$ is an ordinary augmentation of $\mu$ \cite[Thm. 3.18]{Vaq3}.

For the proof of Theorem \ref{small}, we will need some preliminary results.

\subsection{Ubication of roots of key polynomials}

For some pair $(a,\dta)\in\kb\times\La$, let $\overline \mu=\vad$ and  $\mu=\resk(\overline \mu)$. For every $c\in\kb$ and $f\in\kx$, the following equivalences  were shown in \cite[Lemma 2.4]{Rig} and \cite[Proposition 3.3]{Vaq3}, respectively:
\begin{equation}\label{unitOm}
\begin{array}{rcl}
	\ino(x-c) \ \mbox{ is a  unit in }\ \ggo&\Longleftrightarrow& \overline \mu(x-c)<\dta.\\
		\ino f \ \mbox{ is a  unit in }\ \ggo&\Longleftrightarrow& \inm f \ \mbox{ is a  unit in }\ \ggm.
\end{array}
\end{equation}

As mentioned in Section \ref{subsecMP}, every valuation-transcen\-dental  $\mu\in\vv$ admits pairs $(a,\dta)\in\kb\times\La$ such that $\mu=\resk(\vad)$.
All these pairs have a common value of $\dta$. Indeed, by Theorem \ref{nov}, we have
\[
\mu=\resk(\vad)\imp \dta=\max\left\{\vad(x-c)\mid c\in\kb\right\}=\delta(\mu).
\]

The following result gives some information about  the possible ``centers" $a\in\kb$.

\begin{Prop}\label{centers}
Let $\mu\in\vv$ be valuation-transcendental and denote $\dta=\dta(\mu)$. Then, for every $f\in\kx$ the following conditions are equivalent. 	
\begin{enumerate}
	\item[(a)] $\inm f$ is not a unit in $\ggm$.
	\item[(b)] There exists $a\in\op{Z}(f)$ such that $\resk(\vad)=\mu$.
	\item[(c)] $\epm(f)=\dta$.
\end{enumerate}	
\end{Prop}

\begin{proof}
	Take any $\overline \mu=v_{b,\dta}$ such that $\resk(\overline \mu)=\mu$. 
	
	Let us prove that (a) $\Rightarrow$ (b). Suppose that $\inm f$ is not a unit in $\ggm$. By (\ref{unitOm}), $\ino f$ is not a unit in $\ggo$. Hence, there exists  $a\in\op{Z}(f)$ such that $\ino(x-a)$ is not a unit in $\ggo$; thus,  $\overline \mu(x-a)=\dta$, again by (\ref{unitOm}). Hence, $\overline \mu=\vad$.

The implication (b) $\Rightarrow$ (c) follows immediately from Theorem \ref{nov}. 

Finally, let us show that (c) $\Rightarrow$ (a). If $\epm(f)=\dta$, then $\overline \mu(x-a)=\dta$ for some $a\in\op{Z}(f)$, by Theorem \ref{nov}. By (\ref{unitOm}), we deduce that $\ino(x-a)$ is not a unit in $\ggo$,   $\ino f$ is not a unit in $\ggo$, and  $\inm f$ is not a unit in $\ggm$.
\end{proof}

\begin{Cor}\label{epep}
	If $\mu<\eta$ in $\vv$ are both valuation-transcendental, then $\dta(\mu)<\dta(\eta)$. 	
\end{Cor}

\begin{proof}
Take any $\chi\in\tme$. By Proposition \ref{centers}, $\epm(\chi)=\dta(\mu)$. Since $\mu(\chi)<\eta(\chi)$ and $\mu(f)=\eta(f)$ for all nonzero $f\in\kx$ of smaller degree, we conclude that $\dta(\mu)=\epm(\chi)<\ep_\eta(\chi)\le \dta(\eta)$. 
\end{proof}

In terms of a fixed pair $(a,\dta)$ such that $\mu=\resk \vad$, the following characterization of the property $\epm(f)<\dta$ is very useful too. 

\begin{Cor}\label{classic}
	Let  $\mu=\resk(v_{a,\delta})$ for some $(a,\dta)\in\kb\times\La$ and take $f\in K[x]$. The initial coefficient $\inm f$ is a unit in $\ggm$ if and only if $Z(f)\cap B(a,\delta)=\emptyset$. Moreover, if either one of these conditions holds, then $\mu(f)=v(f(a))$.
\end{Cor}

\begin{proof}
 Let $\omu=\vad$.   If $\inm f$ is a unit in $\ggm$, then $\epm(g)<\dta$ by Proposition \ref{centers}. For every $\alpha\in Z(f)$, we have $\al\not\in B(a,\delta)$ because 
    $$v(a-\alpha)=\omu(a-x+x-\alpha)=\omu(x-\alpha)<\delta,$$
by Theorem \ref{nov}. Conversely, if $Z(f)\cap B(a,\delta)=\emptyset$, then for every  $\alpha\in Z(f)$ we have
    $$\omu(x-\alpha)=\omu(x-a+a-\alpha)=v(a-\alpha)<\delta.$$
    Thus, $\epm(f)<\dta$ and Proposition \ref{centers} shows that  $\inm f$ is a unit in $\ggm$. 
    For the last claim, it follows from above that
	\[
	\mu(f)=\sum_{b\in\op{Z}(f)}\omu(x-b)=\sum_{b\in\op{Z}(f)}v(a-b)=v(f(a)).
	\]
\end{proof}

In the Henselian case, this result was proved in \cite[Proposition 3.16]{Vaq3}.
Note that, if $\epm(f)<\dta$, then the value of $f(b)$ is constant for every $b$ contained in the ball $B(a,\delta)$. 

\subsection{Rigidity}

The valuation $v$ on $\kb$ determines a Henselization $(\kh,v^h)$ of $(K,v)$. 
In the Galois group $G={\rm Aut}(\kb/K)$, consider the decomposition subgroup
$$
\mathcal D=\left\{\sg\in G\mid v\circ\sg=v\right\}.
$$

Let $\ks$ be the separable closure of $K$ in $\kb$.
The field $K^h\sub \ks$ is defined as the fixed field of the restriction to $\ks$ of the decomposition group.
The valuation $\vh$ is just the restriction of $v$ to $\kh$.

Let $\vv^h=\vv(\vh,\La)$ be the tree of all $\La$-valued extensions of $\vh$ to $K^h(x)$. The  ``rigidity theorem" \cite[Thm. A]{Rig} is essential for our purpose.

\begin{Teo}\label{rigidity}
The restriction mapping
$\vv^h\to\vv$, defined by $\rho\mapsto \rho_{\mid \kx}$,
is an isomorphism of posets. 
\end{Teo}


The group $\dd=\op{Aut}(\kb/K^h)$ has a natural action on $\kbx$, just by acting on the coefficients of the polynomials. Since  $v\circ \sg=v$, for all $\sg\in\dd$, we deduce the following action on the valuation-transcendental nodes of $\tb$:
\begin{equation}\label{action D}
\vad\circ \sg^{-1}=v_{\sg(a),\dta}\quad\mbox{for all}\quad(a,\dta)\in\kb\times \La,\ \sg\in\dd.     
\end{equation}
Clearly, all pairs in the same $\dd$-orbit have the same restriction to $\khx$. We deduce from Theorem \ref{rigidity} and Proposition \ref{centers} the following criterion.

\begin{Prop}\label{equalRes}
Let  $a,b\in\kb$ and $\dta\in\La$. Then,
\[
 \resk(\vad)=\resk(v_{b,\dta})\sii \sg(b)\in B(a,\dta)\ \mbox{ for some }\ \sg\in\dd.
\]
\end{Prop}

\begin{proof}
By Theorem \ref{rigidity}, we can assume that $(K,v)$ is Henselian; under this assumption, $\mathcal{D}$ is precisely the group $\text{Aut}(\overline K/K)$.

Suppose that $\mu:=\resk(\vad)=\resk(v_{b,\dta})$ and let $f=p_K(b)$.  Since $\epm(f)=\dta$, Corollary \ref{classic} shows that $\sg(b)\in\bad$, for some $\sg\in\dd$. 

For the opposite implication, suppose that $\sg(b)\in\bad$ for some $\sigma\in\mathcal{D}$. Then $v_{a,\dta}=v_{\sg(b),\dta}=v_{b,\dta}\circ\sigma^{-1}$ (where the last equality holds by \eqref{action D}). Hence, $v_{a,\dta}$ and $v_{b,\dta}$ have the same restriction to $K[x]$.
\end{proof}

\begin{Cor}
Let $f\in\kpm$ for a valuation-transcendental $\mu\in\vv$. The set of optimizing roots of $f$ is equal to $\op{Z}(F)$ for some irreducible factor $F$ of $f$ in $\khx$.	
\end{Cor}

\begin{proof}
By Proposition \ref{centers}, $\epm(f)=\dta(\mu)$ because $\inm f$ is a prime element in $\ggm$. 

By Theorem \ref{rigidity}, $\mu$ has a unique extension $\mu^h$ to $\khx$. Thus, $\dta(\mu^h)=\dta(\mu)$ is the common radius of all the extensions of $\mu$ and $\mu^h$ to $\kbx$.

By \cite[Proposition 5.6]{NN},
there is an irreducible factor $F\in\khx$ of $f$ such that
\[
\inn_{\mu^h}F\in\kp(\mu^h)\quad\mbox{and}\quad \inn_{\mu^h}(f/F) \ \mbox{ is a unit in }\ \mathcal{G}_{\mu^h}. 
\]
Hence, $\ep_{\mu^h}(f/F)<\dta(\mu^h)=\dta(\mu)=\epm(f)$ by Proposition \ref{centers} applied to $\mu^h$. This implies that
no root of $f/F$ is optimizing root of $f$. 
Thus, the optimizing roots of $f$ are all included in $\op{Z}(F)$. Since the decomposition group acts transitively on $\op{Z}(F)$, Proposition \ref{equalRes} shows that all roots of $F$ are optimizing roots of $f$.   
\end{proof}


We can summarize these arguments as follows.

\begin{Obs}
Let $\mu\in\vv$ be valuation-transcendental. There are a finite number of valuations $\vad$ on $\kbx$ such that $\resk(\vad)=\mu$. The decomposition group $\dd$ acts transitively on them. These valuations can be obtained from any $f\in\kpm$ as follows: take $\dta=\epm(f)$  and  all $a\in\op{Z}(F)$, where $F$ is the irreducible factor of $f$ in $\khx$ containing all optimizing roots of $f$.
\end{Obs}

\subsection{Proof of Theorem \ref{small} when $\eta$ is valuation-transcendental}

By Theorem \ref{rigidity}, we may assume that $(K,v)$ is Henselian. Denote $\dta=\dta(\mu)$ and $\dta'=\dta(\eta)$.

Take any $\phi\in\kp(\eta)$ and any $\chi\in\tme$. Since $\inn_{\eta}\phi$ is a prime element, it is not a unit in $\gge$. Hence, the two statements of  Lemma \ref{basicTd} (ii) show that $\mu(\phi)<\eta(\phi)$ and $\inm \chi\mid \inm \phi$ in $\ggm$. Since $\inm \chi$ is a prime element, $\inm\phi$ is not a unit in $\ggm$ either.
By Proposition \ref{centers},  $\mu=\resk(\vad)$ and $\eta=\resk(v_{b,\dta'})$, for some $a,b\in\op{Z}(\phi)$. Since $\phi$ is irreducible, we have $b=\sg(a)$ for some $\sg\in \dd$.

The valuation $\overline{\eta}=v_{a,\dta'}$ has the properties we are looking for.  By Corollary \ref{epep}, $\dta<\dta'$, so that $\omu<\overline{\eta}$. Finally, $\resk(\overline{\eta})=\eta$ by Proposition \ref{equalRes}.\hfill{$\Box$}\e

The following result is a consequence of Theorem \ref{small} and Proposition \ref{equalRes}.
 
\begin{Cor}\label{unequal}
Let  $\mu,\eta\in\vv$ be valuation-transcendental such that
$	\mu=\resk(v_{a,\delta})$ and $\eta=\resk(v_{b,\ep})$, for some pairs $(a,\delta),(b,\ep)\in \kb\times \La$.
	Then
	\[
	\mu\leq \eta\sii \delta\leq \ep\ \mbox{ and }\ \sg(b)\in B(a,\delta)\ \mbox{ for some }\sigma\in \mathcal D. 
	\] 
\end{Cor}


\subsection{Lifting increasing families of valuations}\label{subsecLifting}

\begin{Def}
	Let $\mathcal C=\{\mu_i\}_{i\in I}$ be an increasing family of valuations in  $\vv$. We say that  an increasing family of valuations $\{\overline \mu_j\}_{j\in J}$ in $\tb$ is an extension of $\mathcal C$ if
	\[
	I=J\quad\mbox{and}\quad \resk(\overline\mu_i)=\mu_i\ \mbox{ for all }i\in I.
	\]
\end{Def}

\begin{Teo}\label{teoextanuchainvalu}
		Let $\mathcal C=\{\mu_i\}_{i\in I}$ be an increasing family of valuations in  $\vv$, all of them valuation-transcendental. For a given $i_0\in I$, fix $\,\omu_{i_0}\in\tb$ such that $\resk(\overline\mu_{i_0})=\mu_{i_0}$. Then, there exists an increasing family of valuations $\,\overline{\mathcal C}=\{\overline \mu_i\}_{i\in I}$ in $\tb$ extending $\,\mathcal C$.
\end{Teo}

\begin{proof}
	By Theorem \ref{rigidity}, we may assume that $(K,v)$ is Henselian.
	
	Consider the set $\mathcal F$ formed by pairs $\left(J,\{\omu_j\}_{j\in J}\right)$ where $i_0\in J\subseteq I$ and $\{\omu_j\}_{j\in J}$ is an increasing family of valuations extending $\{\mu_j\}_{j\in J}$. Define an ordering on $\mathcal F$ in the obvious way:
	\[
	\left(J,\{\omu_j\}_{j\in J}\right)\leq \left(J',\{\omu'_j\}_{j\in J'}\right)
	\Llr J\subseteq J'\ \mbox{ and }\ \omu_j=\omu'_j\ \ \forall j\in J. 
	\]
	
	We have $\left(\{i_0\},\{\overline\mu_{i_0}\}\right)\in \mathcal F$, so $\mathcal F\neq \emptyset$. Moreover, it is easy to check that every totally ordered subset of $\mathcal F$ admits a majorant (taking the union over the sets $J$'s in the family). Hence, by Zorn's Lemma, $\mathcal F$ admits maximal elements. Let us show that if $F_{\max}=\left(J,\{\overline \mu_j\}_{j\in J}\right)$ is maximal, then $I=J$. This will conclude the proof.
	
	For all $j\in J$, write $\delta_j:=\dta(\mu_j)$ and $\omu_j=v_{a_j,\delta_j}$ for some $a_j\in\kb$. 	Suppose that $J\neq I$ and take $i\in I\setminus J$. Suppose that $i<j$ for some $j\in J$. For each $k\in J$, $k\leq j$, we have $\omu_k=v_{a_j,\delta_k}$. Define $\omu_i=v_{a_j,\dta(\mu_i)}$. Setting $J'=J\cup\{i\}$ we have $\left(J', \{\omu_j\}_{j\in J'}\right)\in\mathcal F$ and 
	\[
	\left(J, \{\omu_j\}_{j\in J}\right)<\left(J', \{\omu_j\}_{j\in J'}\right)
	\] 
	and this is a contradiction to the maximality of $F_{\max}$ in $\mathcal F$.
	
	Now, suppose that $J<i$. Consider any $\phi\in\kp(\mu_i)$. Then, $\inn_{\mu_i}\phi$ is not a unit in $\mathcal{G}_{\mu_i}$ and, as argued in the proof of Theorem \ref{small},  $\inn_{\mu_j}\phi$ is not a unit in $\mathcal{G}_{\mu_j}$ either, for all $j\in J$. 
	By Corollary \ref{classic}, each ball $B(a_j,\dta_j)$ contains some root of $\phi$. Since $\phi$ has a finite number of roots and these balls form a descending chain, there is some $a\in\op{Z}(\phi)$ belonging to the intersection of all these balls. The valuation $\omu_i=v_{a,\dta(\mu_i)}$ restricts to $\mu_i$ by Proposition \ref{centers},
	and clearly satisfies $\omu_i>\omu_j$ for all $j\in J$. This contradicts the maximality of $F_{\max}$.  
\end{proof}

\subsection{Proof of Theorem \ref{small} when $\eta$ is valuation-algebraic}

If $\eta$ is valuation-algebraic, then it is the stable limit $\eta=\lim_{i\in I} \mu_i$ of some increasing family $\{\mu_i\}_{i\in I}$ of inner nodes in $\vv$.
Since $\vv$ is a tree and $\mu,\mu_i<\eta$, these valuations are comparable. Now, $\mu>\mu_i$ for all $i$ would imply $\eta=\lim_{i\in A} \mu_i\le\mu$, against our assumption. Thus, there exists some $i_0\in A$ such that $\mu<\mu_i<\eta$ for all $i \ge i_0$.

By Theorem \ref{small} for the valuation-transcendental case, there exists some $\omu_{i_0}$ extending $\mu_{i_0}$ such that $\omu<\omu_{i_0}$. Therefore, we can assume that $\mu=\mu_{i_0}$. 

By Theorem \ref{teoextanuchainvalu}, there  exists an increasing family $\{\omu_i\}_{i\ge i_0}$ of valuations in $\tb$  extending the increasing family  $\{\mu_i\}_{i\ge i_0}$. 
Let $\omu_i=v_{a_i,\dta_i}$ for all $i\ge i_0$. 
We claim that the corresponding ultrametric balls have empty intersection:
\[
\bigcap\nolimits_{i\ge i_0}B(a_i,\dta_i)=\emptyset.
\]
Indeed, suppose that  $a\in\kb$ belongs to $B(a_i,\dta_i)$ for all $i\ge i_0$. This means $v(a-a_i)\ge \dta_i$, or equivalently, $\omu_i(x-a)=\dta_i$. Hence,
\[
\omu_i(x-a)=\dta_i<\dta_j=\omu_j(x-a),\quad\mbox{ for all }j>i\ge i_0. 
\]
Hence, $g:=p_K(a)$ satisfies $\mu_i(g)<\mu_j(g)$ for all $j>i\ge i_0$, contradicting the fact that  $\left\{\mu_i\right\}_{i\in I}$ has a stable limit. 

Therefore, $\{\omu_i\}_{i\ge i_0}$ has a stable limit $\overline\eta=\lim_{i\ge i_0}\omu_i\in \tb$ \cite[Thm. 3.23]{Vaq3}. By construction, $\omu_i<\overline\eta$ for all $i\ge i_0$. On the other hand, for all $f\in\kx$ there is a sufficiently large index $j$ such that
\[
\overline\eta(f)=\omu_j(f)=\mu_j(f)=\eta(f), 
\]
so that $\resk(\overline\eta)=\eta$. \hfill{$\Box$}

\section{Abstract key polynomials}\label{secAKP}

For some $\mu\in\vv$, let us fix an extension  $\omu$ to $\kbx$. 
Let us describe a certain chain $\bb(\omu)$ of balls ordered by decreasing inclusion, determined by $\omu$.

If $\mu$ is valuation-transcendental, then $\omu=\vad$ for some pair $(a,\dta)\in\kb\times\La$. In this case, we  denote 
\[
\bb=\bb(\omu):=\left\{B\in \mathbb B\mid B\supseteq \bad\right\}.
\]

 If $\mu$ is  valuation-algebraic, then it is the stable limit of some increasing family $\cc=\{\mu_i\}_{i\in I}$. As we saw in Section  \ref{subsecLifting},  $\omu$ is the stable limit of some extension $\{\omu_i\}_{i\in I}$  of $\cc$ to $\kbx$. 
  For all $i\in I$, let $\omu_i=v_{a_i,\dta_i}$ and denote $B_i=B(a_i,\dta_i)$. In this case, we  denote
 \[
 \bb=\bb(\omu):=\left\{B\in \mathbb B\mid B\supseteq B_i\mbox{ for some }i\in I\right\}.
 \]

\subsection{Construction of all abstract key polynomials}\label{subsecAllAKP}
The goal of this section is to express all abstract key polynomials for $\mu$ as  minimal polynomials of certain centers of balls in $\mathbb B$.

To this end, the following easy observation will be useful.

\begin{Lem}\label{close}
	Let $Q$ be an abstract key polynomial for $\mu\in\vv$.
	Suppose that $\mu<\eta$ and $\mu(Q)=\eta(Q)$, for some $\eta\in\vv$.  Then,  $Q$ is an abstract key polynomial for $\eta$. 
\end{Lem}

\begin{proof}
	By Lemma \ref{maximal}, $Q\in\kp(\mu_Q)$ is a key polynomial of minimal degree for $\mu_Q$. Thus, $\deg Q=\deg(\mu_Q)\le \deg(\mu)$. Hence, 
	the key polynomials  in $\ty(\mu,\eta)$ have degree greater than or equal to $\deg Q$.
	By the definition of $\ty(\mu,\eta)$, we have
	\[
	\mu(f)=\eta(f) \quad \mbox{ for all }f\in\kx \ \mbox{such that}\ \deg f<\deg Q.
	\]
	In particular, $\mu\left(\partial_sQ\right)=\eta\left(\partial_sQ\right)$ and $\mu\left(\partial_sf\right)=\eta\left(\partial_sf\right)$, for all $s\in \N$. Therefore, 
	\[
	\epsilon_\eta(f)=\epm(f)<\epm(Q)=\epsilon_\eta(Q). 
	\]
	This proves that $Q$ is an abstract key polynomial for $\eta$.
\end{proof}

\begin{Teo}\label{abskeypol1}
For every $\mu\in\vv$ we have
\[
\Psi(\mu)=\bigcup_{B\in \bb}p_K\left(\mink B\right).
\]
\end{Teo}

\begin{proof}
Let us first assume that $\mu$ is valuation-transcendental.

For any $Q\in \Psi(\mu)$, let $\dta'=\epm(Q)$ and take an optimizing root $c$ of $Q$; that is, $v_{b,\dta}(x-c)=\dta'\le\dta$, for some 
extension $v_{b,\dta}$ of $\mu$ to $\kbx$.

By Proposition \ref{equalRes}, we have $\sg(b)\in \bad$ for some $\sg\in\dd$. Hence,
\[
B:=B(\sg(c),\dta')\supseteq B(\sg(b),\dta)=\bad, 
\]
so that $B\in\bb$. Now, since $\sg(c)$ is an optimization root of $Q$ too, Theorem \ref{mp} shows that $\sg(c)\in \mink B$. This proves that
\begin{equation}\label{inclusion}
\Psi(\mu)\subseteq\bigcup_{B\in\bb}p_K\left(\mink B\right).
\end{equation}

For the converse, take any ball $B\in \mathbb B$ with $B=B(a',\delta')\supseteq \bad$. For any $c\in \mink B$ we claim that $Q:=p_K(c)$ is an abstract key polynomial for $\mu$. Indeed, let us show that for every $f\in K[x]$, the condition $\deg f<\deg Q$ implies $\epm(f)<\epm(Q)$.

Since $a\in B$, we have $B=B(a,\delta')$. Since $\delta'\leq \delta$, for every $b\in \overline K$ we have
\begin{equation}\label{eqepsilon}
b\notin B\Llr v(b-a)<\delta'\Llr v_{a,\delta}(x-b)<\delta'.
\end{equation}
Since $c\in \mink B$, the condition $\deg f<\deg Q$ implies  that $f$ has no root in $B$. In particular, for every root $b$ of $f$, we have $v_{a,\delta}(x-b)<\delta'$,  by \eqref{eqepsilon}. Since $\mu=\resk(v_{a,\delta})$ and $v(a-c)\ge\dta'$, Theorem \ref{nov} shows that
\[
\epm(f)< \delta'\le v_{a,\delta}(x-c)\leq \epm(Q)
\]
and this concludes the proof.\e

Let us now deal with the valuation-algebraic case.

Let $Q$ be an abstract key polynomial for $\mu$. Since $Q$ is an abstract key polynomial for $\mu_Q$, Lemma \ref{close} shows that  $Q$ is an abstract  key polynomial for $\mu_i$ for a sufficiently large $i\in I$ for which $\mu_Q<\mu_i<\mu$.

By the valuation-transcendental case, $Q=p_K(c)$ for some $c\in \mink B$, for some ball $B$ containing $B_i$. Since $B\in\bb$, this proves the inclusion (\ref{inclusion}).

For the converse, take any ball $B\in \mathbb B$ with $B\supseteq B_i$ for some $i\in I$, and take $Q=p_K(c)$ for some $c\in \mink B$. Take $j\in I$, $j>i$, large enough to stabilize $Q$; that is, $\mu_j(Q)=\mu(Q)$.
By  the valuation-transcendental case, $Q$ is an abstract key polynomial for $\mu_j$. By Lemma \ref{close}, $Q$ is an abstract key polynomial for $\mu$.
\end{proof}

Finally, let us recall  \cite[Theorem 5.3]{Andrei}, which describes the balls associated to the truncated valuation $\mu_Q$, for every abstract key polynomial $Q$ for $\mu$.

\begin{Prop}\label{BS}
For any $\mu\in\vv$, take  
$Q=p_K(c)$ for some $c\in\mink B$, for some $B\in\bb$.  
If  $\omu(x-c)=\epm(Q)$, then $\mu_Q=\resk(v_{c,\epm(Q)})$.
\end{Prop}

\subsection{Construction of a complete set of abstract key polynomials}\label{subsecCompleteseq}

Consider a set  $\ss\sub\Psi=\Psi(\mu)$  of abstract key polynomials for $\mu$. We say that $\ss$ is \textbf{complete} if for every $f\in\kx\setminus K$, there exists $Q\in\ss$ such that
\[
\deg Q\le \deg f \quad\mbox{and}\quad \mu_Q(f)=\mu(f).
\] 
The existence of complete sets of abstract key polynomials was proved in \cite{NS2018}, where it is described too, how to construct some sort of ``minimal" complete sets.

 For all $m\in\N$, let $\Psi_m\sub\Psi$ be the  subset  of polynomials of degree $m$.  
 A complete set $\ss$ can be constructed as follows. First, consider a pre-ordering on $\Psi$ according to the $\epm$-values:
 \[
 Q\le Q'\sii \epm(Q)\le \epm(Q').
 \]
Then, inside  each $\Psi_m\ne\emptyset$, take any totally ordered cofinal subset $\ss_m\sub\Psi_m$. We then set $\ss=\bigcup_m\ss_m$.

The aim of this section is to reinterpret these results in terms of ultrametric balls. To this end, we introduce a relevant concept.

\begin{LDef}\label{optimal}
	We say that $B=B(b,\ga)\in\bb$ is a $\mu$-\textbf{optimal} ball if it satisfies the following equivalent conditions
	\begin{enumerate}
	\item [(a)] There exists $c\in\mink B$ such that \,$\resk(v_{b,\ga})=\mu_Q$, where $Q=p_K(c)$.
	\item [(b)]  There exists $c\in\mink B$ such that \,$\epm(Q)=\ga$, where $Q=p_K(c)$.
\end{enumerate}	
\end{LDef}

\begin{proof}
Let  us show that (a) $\Rightarrow$ (b). Since  $\resk(v_{c,\ga})=\resk(v_{b,\ga})=\mu_Q$, we have $\ga=\ep_{\mu_Q}(Q)$ by Theorem \ref{nov}. On the other hand, $\epm(Q)=\ep_{\mu_Q}(Q)$.

Let  us show that (b) $\Rightarrow$ (a). 
Let $\overline{\nu}=v_{b,\ga}$. Since $B\in\bb$, we have $\overline{\nu}\le\omu$.
The condition $\epm(Q)=\ga$ implies $\omu(x-c)\le\ga$, by Theorem \ref{nov}. 
Since $v(c-b)\ge \ga$, we deduce that $\ga=\overline{\nu}(x-c)\le \omu(x-c)$. Therefore, $\omu(x-c)=\ga$ and Proposition \ref{BS} shows that $\resk(v_{c,\ga})=\mu_Q$.   Since $v_{c,\ga}=v_{b,\ga}$, we are done.
\end{proof}


Let  us denote the subset of all $\mu$-optimal balls by
\[
\bopt=\bopt(\omu)\sub\bb. 
\]

For any $\mu$-optimal $B=B(b,\ga)$ we define the set of \textbf{optimal elements} in $B$ as:
\[
\opt(B):=\{c\in \mink B\mid \epm(Q)=\ga, \mbox{ where  }Q=p_K(c)\}.
\]
By Lemma-Definition \ref{optimal}, if $B,B'\in\bopt$, we have
\[
 p_K\left(\opt(B)\right)\cap  p_K\left(\opt(B')\right)\ne\emptyset
\imp B=B',
\]
because the two balls have the same radius.

For $m\in\N$, we denote the subsets of those balls $B$ such that $\deg_KB=m$ by
\[
\bb_m\sub\bb,\qquad \bopt_m\sub\bopt.
\]

\begin{Lem}\label{cofinal}
Let  $m\in\N$ such that $\bb_m\ne\emptyset$. Then,  $\bopt_m$ is a cofinal family in $\bb_m$, with respect to descending inclusion.
\end{Lem}

\begin{proof}
Suppose that $B=B(b,\ep)\in\bb_m$ is not $\mu$-optimal. Take any  $c\in\mink B$ and let $Q=p_K(c)$, $\ga=\epm(Q)$. Clearly, 
\[
\ep=v_{b,\ep}(x-c)\le\omu(x-c)\le\ga. 
\]
Since $B$ in not $\mu$-optimal, we have necessarily $\ep<\ga$. By Theorem \ref{nov}, there exists some $c'\in \op{Z}(Q)$ such that $\omu(x-c')=\ga$. Consider the ball $B':=B(c',\ga)$. The proof of the lemma will be complete if we show that $B'$ satisfies:
\begin{equation}\label{aims}
B'\in\bb\quad\mbox{and}\quad B\supsetneq B'.
\end{equation}
Indeed, $B\supsetneq B'$ implies $\deg_KB\le \deg_KB'$. Since $\deg_Kc'=m$, we deduce that $\deg_KB'=m$ and $c'\in\mink B'$. Thus, $B'$ is $\mu$-optimal because $c'$ satisfies the condition (b) of Lemma-Definition \ref{optimal}. 

Let us first prove (\ref{aims}) when $\mu$ is valuation-transcendental. In this case, $\bb$ is the set of closed balls containing $\bad$, for some pair $(a,\dta)\in\kb\times\La$ such that $\omu=v_{a,\dta}$.  Since, $B\supseteq \bad$, we have $B=B(a,\ep)$. Since
$\vad(x-c')=\ga\le\dta$, we see that
\[
v(a-c')\ge\ga>\ep.
\]
Hence, $c'\in B$ and $a\in B'$; or equivalently,  $B\supsetneq B'\supseteq \bad$. This proves (\ref{aims}). 

Now, suppose that $\mu$ is valuation-algebraic. Then, $\omu$ is the stable limit of some increasing family $\left\{\omu_i\right\}_{i\in I}$. For all $i\in I$, let 
$B_i=B(a_i,\dta_i)$, where $\omu_i=v_{a_i,\dta_i}$.
We may take $i\in I$ large enough to have $B\supseteq B_i$ and $\ep_{\mu_i}(Q)=\epm(Q)$, simultaneously. Then, the above arguments show that $B\supsetneq B'\supseteq B_i$. This proves (\ref{aims}) too.
\end{proof}

\begin{Cor}\label{lastBall}
	If $\bb_m$ has a last (smallest) ball $B$, then $B$ is $\mu$-optimal. 
\end{Cor}

As a consequence of Lemma \ref{cofinal}, we may strengthen Theorem \ref{abskeypol1}  as follows.

\begin{Teo}\label{abskeypol}
Let $\mu\in\vv$. We can express $\Psi(\mu)$ as the disjoint union:
	\[
	\Psi(\mu)=\bigcup_{B\in \bopt}p_K\left(\opt(B)\right).
	\]
\end{Teo}

Now, it is easy to extract a minimal complete set of abstract key polynomials for $\mu$ as follows.  Consider the sequence of all degrees of the balls in $\bb$:
\[
 1=m_0<m_1<\cdots<m_n<\cdots 
\]

This sequence is finite if $\deg(\mu)$ is finite, the maximal degree being $\deg(\mu)$ in this case. The sequence is infinite only when 
 $\mu$ is the stable limit of an increasing family $\{\mu_i\}_{i \in I}$ such that $\deg(\mu_i)$ are unbounded.

We may split $\Psi(\mu)$ as the disjoint union $\bigcup_{n}\Psi_{m_n}$.
Take for each $m_n$ a cofinal family of optimal balls of $\bb_{m_n}$. For each optimal ball $B$ in this family, choose (only) one element $c\in\opt(B)$ and consider the abstract key polynomial $p_K(c)$. 

Let $\ss_{m_n}\sub\Psi_{m_n}$ be the family of abstract key polynomials obtained in this way. Since the $\epm$-values of these abstract key polynomials are all different, the family $\ss_{m_n}$ is totally ordered. Thus, $\ss:=\bigcup_n\ss_{m_n}$ is a complete set of abstract key polynomials for $\mu$.

\section{Key polynomials}\label{secKP}
In this section, $\mu\in\vv$ is assumed to be valuation-transcendental; or equivalently, $\kpm\ne\emptyset$ (Theorem \ref{EKP}). The first paper describing a relationship between balls and key polynomials for $\mu$ was \cite{PP}.
 As first examples of such a relationship let us quote the following facts.

\begin{Prop}\label{Ex1}
If $\mu=\resk(v_{a,\delta})$ for some $(a,\delta)\in\overline K\times \La$, then
\begin{itemize}
 \item $\Psi(\mu)\cap {\rm KP}(\mu)=p_K\left(\mink B(a,\delta)\right)$.
 \item $\kpm\subseteq p_K\left(\bad \right)$.
\end{itemize}
\end{Prop}

\begin{proof}
The first item is an immediate consequence of Theorem \ref{mp}. The second item follows from Corollary \ref{classic} and the fact that key polynomials for $\mu$ are prime elements in the graded algebra $\ggm$.  
\end{proof}

The next result is a useful way to characterize key polynomials for $\mu$. It will be the main tool used in this section.  

\begin{Prop}\label{mainresult}
The set of key polynomials for $\mu$ is given by
\[
{\rm KP}(\mu)=\bigcup_{\mu<\eta}\tme.
\]
\end{Prop}

\begin{proof}
By Lemma \ref{basicTd}, the set on the right hand side is contained in the set on the left hand side.
Take any $\phi\in {\rm KP}(\mu)$ and any $\ga\in\g$ such that $\gamma>\mu(\phi)$. Consider the ordinary augmentation $\eta=[\mu;\,\phi,\gamma]$, defined on $\phi$-expansions as
\begin{equation}\label{MacLane ordinary augmentation}
\eta\left(f_0+f_1\phi+\ldots+f_r\phi^r\right):=\min_{0\leq i\leq r} \{\mu(f_i)+i \gamma\}.
\end{equation}
Thus, $\mu(\phi)<\gamma=\eta(\phi)$ and $\mu(f)=\eta(f)$ for all $f\in K[x]$ with $\deg(f)<\deg(\phi)$. Consequently, $\phi\in \textbf{t}(\mu,\eta)$ and this concludes the proof.
\end{proof}

If we let the above union run through a set of representatives of the tangent directions of $\mu$, then 
this union is disjoint by Lemma \ref{disjointTd}. If $\mu=\resk(v_{a,\delta})$ for some $(a,\delta)\in \overline K\times\Lambda$, we will show below that this disjoint union is in one-to-one correspondence with the partition of the ball $B(a,\delta)$ into open balls of radius $\delta$, up to the action of the decomposition group of $v$.

\begin{Lema}\label{varofGiulio}
Take $\mu,\eta\in\vv$ such that $\mu<\eta$. Let $\dta=\dta(\mu)$ and
$\eta=\resk(v_{b,\epsilon})$ for some pair $(b,\ep)\in\kb\times \La$.
Given $f\in K[x]$ we have
\[
\mu(f)<\eta(f)\Llr f\mbox{ admits a root in }B^\circ (b,\dta).
\]
\end{Lema}

\begin{proof}
By Theorem \ref{Nartgoingup}, there exists $a\in\kb$ such that 
$\resk(\vad)=\mu$ and $\vad<v_{b,\epsilon}$. 
It is easy to see that for any element $c\in \overline{K}$ we have
\[
v_{a,\delta}(x-c)<v_{b,\epsilon}(x-c)\Llr c\in B^\circ(b,\delta).
\]
Hence the result follows.
\end{proof}

As a corollary of Lemma \ref{varofGiulio}, we obtain the following result.

\begin{Cor}\label{corGiulio}
Assume that $\mu=\resk(v_{a,\delta})$. For a nonzero $f\in K[x]$, there exists a valuation $\eta\in\vv$ such that $\mu< \eta$ and $\mu(f)<\eta(f)$ if and only if $f$ admits a root in $B^\circ(b,\delta)$ for some $b\in B(a,\delta)$.
\end{Cor}

\begin{proof}
Assume that $\mu(f)<\eta(f)$ for some  $\mu< \eta$ in $\vv$. We can assume, without loss of generality, that $\eta$ is valuation-transcendental. By Theorem \ref{Nartgoingup}, there exists an extension $v_{b,\ep}$ of $\eta$ such that $v_{a,\delta}<v_{b,\ep}$. By the previous lemma, $f$ admits a root in $B^\circ(b,\delta)$; on the other hand, $b\in B(a,\delta)$.

The converse follows from Lemma \ref{varofGiulio} by taking $\eta=\resk(v_{b,v(c-b)})$, where $c$ is a root of $f$ in $B^\circ(b,\delta)$ for some $b\in\bad$. 
\end{proof}

The next result is a characterization of polynomials of ${\rm KP}(\mu)$ as minimal polynomials of suitable elements of $\mathbb B^\circ$.

\begin{Teo}\label{mainthm}
Assume that $\mu=\resk(v_{a,\delta})$ for some $(a,\delta)\in\overline K\times \La$. Set
\[
\mathcal U=\{c\in B(a,\delta)\mid c\in \mink B^\circ(c,\delta)\}.
\]
Then, $\kpm= p_K(\mathcal U)$.
\end{Teo}


\begin{proof}
Take any valuation $\eta$ such that $\mu<\eta$. We can suppose that $\eta$ is valuation-transcendental. By Corollary \ref{corGiulio}, the elements $f\in K[x]$ for which $\mu(f)<\eta(f)$ are exactly those having a root in the ball $B^\circ(b,\delta)$ for some $b\in\bad$. 

Hence, every $\phi\in \tme$  has smallest degree among the polynomials having this property. This and Proposition \ref{mainresult} show that
\[
    \kpm\subseteq p_K(\mathcal U).
\]

Now, take any $c\in \mathcal U$. In particular, $c\in B(a,\delta)$. Fix any $\delta'>\delta$. Then, $
\mu< \eta:=\resk(v_{c,\delta'})$.
By Lemma \ref{varofGiulio}, every nonzero $f\in K[x]$ such that $\mu(f) <\eta(f)$ has a root in $B^\circ(c,\delta)$. Hence, $\deg f\ge\deg_K c$. This shows that $p_K(c)\in \tme\sub\kpm$. Therefore,
$p_K(\mathcal U)\subseteq {\rm KP}(\mu)$
and this concludes the proof.
\end{proof}

As an immediate consequence of the above result we obtain the following.

\begin{Cor}\label{kp decomposition}
	Assume that $\mu=\resk(v_{a,\delta})$ for some $(a,\delta)\in\overline K\times \La$. Let
	\begin{equation}\label{decomposition}
		B(a,\delta)=\bigsqcup_{i\in I}B^\circ(a_i,\delta)  
	\end{equation}
be the partition of $B(a,\delta)$ as a disjoint union of open balls of radius $\delta$. Then,
\begin{equation}\label{disjointKpi}
{\rm KP}(\mu)=\bigcup_{i\in I}{\rm KP}_i(\mu),\qquad \kp_i(\mu):=p_K\left(\mink B^\circ(a_i,\delta)\right).
\end{equation}
\end{Cor}

 We keep the notation of Corollary \ref{kp decomposition}.

\begin{Lem}\label{kpim}
	Take $\mu,\eta\in\vv$ such that $\mu<\eta$. Assume that $\mu=\resk(v_{a,\delta})$ for some $(a,\delta)\in\overline K\times \La$ and 
	$\eta=\resk(v_{b,\epsilon})$ for some $b\in\bad$ and $\ep>\dta$. Let $i\in I$ be uniquely determined by $B^\circ(b,\delta)=B^\circ(a_i,\delta)$. Then 
$$\tme=\kp_i(\mu).$$    
\end{Lem}

\begin{proof}
By definition,  $f\in\mathbf t(\mu,\eta)$ if $\mu(f)<\nu(f)$ and $f$ has minimal degre with this property.
By definition,  $f\in {\rm KP}_i(\mu)$ if it has a root in $B^\circ(a_i,\delta)$ and $f$  has minimal degree with this property.
By Lemma \ref{varofGiulio}, both conditions are equivalent.
\end{proof}

Therefore, in the decomposition (\ref{disjointKpi}), if ${\rm KP}_i(\mu)\cap {\rm KP}_j(\mu)\neq \emptyset$ for some $i,j\in I$, then ${\rm KP}_i(\mu)= {\rm KP}_j(\mu)$. 
Let us see exactly which open balls can yield the same set $\kp_i(\mu)=\tme$. We recall that two key polynomials $f,g$ belong to the same tangent direction of $\mu$ if and only if $f\smu g$ (Lemma \ref{basicTd}). 

\begin{Lem}\label{lemamnartomp}
Let $f=p_K(c)$, $g=p_K(d)$ be the  key polynomials for $\mu$ associated to  $c,d\in \mathcal U$.	Then, $f\smu g$ if and only if $B^\circ(c,\delta)=B^\circ(d,\delta)^\sg$ for some $\sg\in \mathcal D$.
\end{Lem}

\begin{proof}
	Suppose that $B^\circ(c,\delta)=B^\circ(d,\delta)^\sg=B^\circ(\sg(d),\delta)$ for some $\sg\in \mathcal D$.
	Take $\ep:=v(c-\sg(d))>\dta$ and $\eta:=\res(v_{c,\ep})$. By Lemma \ref{kpim}, $f,g\in\tme$, so that $f\smu g$ by Lemma \ref{basicTd}. 
	
	Conversely, suppose  $f\smu g$ and take some  valuation-transcendental $\eta$ on $\kx$ such that $\mu<\eta$ and $f,g\in\tme$. By Theorem \ref{Nartgoingup}, there exist $b\in B(a,\dta)$ and $\ep\in\g$, $\ep>\dta$, such that $\eta=\res(v_{b,\ep})$.
	
	By Lemma \ref{varofGiulio}, $f$ and $g$ have a root in $B^\circ(b,\dta)$. Let $\sg,\tau\in \mathcal D$ such that $\sg(c),\tau(d)\in B^\circ(b,\dta)$. Then,
	\[
	B^\circ(c,\dta)^\sg=B^\circ(b,\dta)=B^\circ(d,\dta)^\tau.
	\]
	This proves that $B^\circ(c,\dta)$ and $B^\circ(d,\dta)$ are conjugate under the action of $\mathcal D$.
\end{proof}

Therefore, the tangent directions of $\mu$ are in one-to-one correspondence with the decomposition \eqref{decomposition}, up to the action of $\mathcal{D}$.
By Lemma \ref{disjointTd}, the tangent directions of $\mu$ parametrize the quotient set $\left(\vv_{>\mu}\right)/\!\sim$, where $\sim$ is the equivalence relation:
\[
\nu\sim\eta \sii (\mu,\nu]\cap (\mu,\eta]\ne\emptyset. 
\]
For $\nu,\eta$ valuation-transcendental, this equivalence relation can be expressed as follows in terms of the decomposition \eqref{decomposition}. Let $\nu=\res(v_{b,\ep})$, $\eta=\res(v_{c,\ep'})$ for some $b,c\in B(a,\dta)$ and $\ep,\ep'>\dta$; then $\nu\sim\eta$ if and only if the balls $B^\circ(b,\dta)$ and $B^\circ(c,\dta)$ are conjugate under the action of $\mathcal{D}$.

\subsection*{The value-transcendental case}
If $\mu$ is value-transcendental, then 
the graded algebra $\ggm$ has a unique homogeneous prime ideal and $\kpm/\!\smu$ is a one-element set \cite[Theorem 4.2]{KP}.
By Lemma \ref{basicTd}, all tangent directions of $\mu$ coincide.

Therefore, the results of this section take a very simple form in this case.
The extensions of $\mu$ to $\kbx$ are of the form $\vad$ with $\dta\in\La\setminus\g$. Hence, 
\[
\bad=B^\circ(a,\dta).
\]  
Thus, Theorem \ref{mainthm} and Corollary \ref{kp decomposition} are coherent with the fact that there is a unique tangent direction of $\mu$.

\section{Limit key polynomials}\label{limitjepol}
In this Section, we consider the problem of characterizing limit key polynomials in terms of elements of $\mathbb B$.

Let $\mathcal C=\{\mu_i\}_{i\in I}$ be an increasing family of valuations in $\vv$ admitting no last element. By Theorem \ref{teoextanuchainvalu}, there exists an increasing  family of valuations $\overline{\mathcal C}=\{\overline\mu_i\}_{i\in I}$ in $\overline{\mathcal V}$ extending $\cc$. 
Thus, 
for each $i\in I$ we have $\omu_i=v_{a_i,\delta_i}$, for some pair $(a_i,\dta_i)\in\kb\times \La$. We obtain in this way a decreasing family of closed balls 
\[
\{B_i\}_{i\in I,\qquad }B_i:=B(a_i,\delta_i)\in \mathbb B.
\]
It is well known \cite[Proposition 2.12]{Vaq3} that 
\[
\kpi(\cc)\ne\emptyset\sii \bigcap_{i\in I}B_i\ne\emptyset.
\]

Let us reproduce a short proof of this result.

\begin{Prop}\label{reprove}
	Let $B=\bigcap_{i\in I}B_i$. A polynomial $f\in\kx$ is $\cc$-unstable if and only if $\,\op{Z}(f)\cap B\ne\emptyset$. 	
\end{Prop}

\begin{proof}
Take $f\in\kx$.	By Lemma \ref{basicTd},(ii) we have
\[
f \ \mbox{ is $\,\cc$-unstable}\sii \inn_{\mu_i}f \ \mbox{ is not a unit in $\mathcal G_{\mu_i}$ for all }i\in I.
\]
Now, Corollary \ref{classic} shows that the right hand side condition is equivalent to $\,\op{Z}(f)\cap B_i\ne\emptyset$ for all $i\in I$. Since $f$ has a finite number of zeros, this is equivalent to $\,\op{Z}(f)\cap B\ne\emptyset\,$ too.
\end{proof}

Therefore, we obtain a result for limit key polynomials which is coherent with Theorem \ref{abskeypol1} for abstract key polynomials  and  Theorem  \ref{mainthm} for key polynomials.

\begin{Teo}\label{mainlimitkp}
Let $B=\bigcap_{i\in I}B_i$. Then, $\kpi(\cc)=p_K(\mink B)$.
\end{Teo}

\begin{proof}
We need only to consider objects of minimal degree satisfying the equivalent conditions of Proposition \ref{reprove}. 
\end{proof}

\begin{Obs} Given an increasing family $\mathcal C=\{\mu_{i}=(v_{a_i,\delta_i})_{|K[x]}\}$ of valuations in $\mathcal V$ as above such that the index set $I$ is well-ordered, if we choose $a_i\in B_i$ so that $v(a_i-a_j)=\delta_i$ for $i<j$, then  $E=\{a_i\}_{i\in I}$ is a pseudo-convergent sequence in $\overline{K}$ (see \cite{Ostr} and  also \cite[p. 283]{APZ2} for this approach) and $B$ is precisely the set of pseudo-limits  of $E$ in $\overline{K}$. The sequence $E$ is either of algebraic type or of transcendental type in the sense of Kaplansky \cite{Kap}, if either $B$ is non-empty or not, respectively \cite{PS,PS2}.  If $B\not=\emptyset$, then, by Proposition \ref{reprove}, the $\mathcal{C}$-unstable polynomials $f\in K[x]$ are those having a root which is a pseudo-limit of $E$ and, by Theorem \ref{mainlimitkp}, the limit key polynomials for $\mathcal{C}$ are the minimal polynomials of some pseudo-limit of $E$ having minimal degree. If $B=\emptyset$ and thus every $f\in K[x]$ is $\mathcal{C}$-stable,  Corollary \ref{classic} shows that, for each such an $f$, we have $\mu_i(f)=v(f(a_i))$ for all $i\in I$ sufficiently large. Therefore, $\mu=\lim(\mathcal{C})$ as defined in \eqref{limit C} is precisely equal to the valuation $v_E$ associated to $E$, defined as $v_E(f)=v(f(a_i))$ for all $i$ sufficiently large (see \cite{PS,PS2} for other relevant results regarding the valuation $v_E$).
\end{Obs}

\begin{Obs}
The use of decreasing chain of ultrametric balls (under the name of approximation types) have been used to treat similar problems as the one is this section (see \cite{KuhlAT}, \cite{NartJosnei} and \cite{CaioJosnei}).  
\end{Obs}
\section{On Mac Lane--Vaqui\'e chains}\label{secMLV}

A \textbf{Mac Lane-Vaqui\'e chain} (abbreviated MLV chain) of a given
 $\mu\in \vv$ is a finite, or countably infinite, totally ordered sequence in $\vv$ upper bounded by $\mu$:
\[
\mu_0\,<\,\mu_1\,<\,\cdots\,<\,\mu_n\,<\,\cdots\le\mu
\]	
satisfying the following conditions:\e

(MLV1) \quad $1=\deg(\mu_0)\,<\,\deg(\mu_1)\,<\,\cdots\,<\,\deg(\mu_n)\,<\,\cdots$.\e

(MLV2) \ Each step $\mu_{n-1}<\mu_n$ is an augmentation of one of the following kinds: 

$\bullet$ \ \textbf{ordinary augmentation}: $\mu_n=[\mu_{n-1};\,\phi_n,\ga_n]$, for some $\phi_n\in\kp(\mu_{n-1})$ and some $\ga_n\in\La$, $\ga_n>\mu_{n-1}(\phi_n)$ (see the definition in \eqref{MacLane ordinary augmentation}).

$\bullet$ \ \textbf{limit augmentation}: there is a well-ordered increasing family  $\cc_{n-1}=\{\rho_i\}_{i\in I_{n-1}}$ of valuations of constant degree, such that $\mu_{n-1}=\min\left(\cc_{n-1}\right)$ and 
$\mu_n$ is the limit augmentation: \,$\mu_n=[\cc_{n-1};\,\phi_n,\ga_n]$, for some $\phi_n\in\kpi(\cc_{n-1})$ and some $\ga_n\in\La$, $\ga_n>\rho_i(\phi_n)$ for all $i\in I_{n-1}$.\e

(MLV3) \ $ \phi_n\not\in\ty(\mu_n,\mu)$  for all $\mu_n<\mu$.\e

(MLV4) \ If $\mu$ is valuation-transcendental, then the sequence  has $\mu$ as its last valuation: 
\begin{equation}\label{depth}
\mu_0\,<\,\mu_1\,<\,\cdots\,<\,\mu_r=\mu.
\end{equation}	

If $\mu$ is valuation-algebraic of finite degree, then the sequence  has a last valuation (say) $\mu_{r-1}$ and there is a well-ordered increasing family  $\cc_{r-1}$ of valuations of constant degree,  such that $\mu_{r-1}=\min\left(\cc_{r-1}\right)$ and 
$\mu$ is the stable limit of $\cc_{r-1}$: 
\begin{equation}\label{depth2}
\mu_0\,<\,\mu_1\,<\,\cdots\,<\,\mu_{r-1}\,<\,\mu=\lim\left(\,\cc_{r-1}\right).
\end{equation}	

If $\mu$ is valuation-algebraic of infinite degree, then the sequence is infinite and  $\mu=\lim_{n\in\N}\mu_n$ is the stable limit of the sequence.\e

The existence of such chains was proved by Mac Lane in the discrete rank-one case, and by Vaqui\'e in the general case \cite{mcla,Vaq}. Their relevance lies in the fact that they contain intrinsic information about $\mu$. For instance, in \cite[Section 4]{MLV} it is proved that the following data are common to all MLV chains of the same $\mu$:\e

$\bullet$ \ The length of the chain, called the \textbf{depth} of $\mu$. A valuation $\mu$ as in (\ref{depth}) or  (\ref{depth2}) has depth $r$, while a valuation-algebraic $\mu$ of infinite degree has infinite depth.

$\bullet$ \ The sequence of degrees \ $1=\deg(\mu_0)\,<\,\deg(\mu_1)\,<\,\cdots\,<\,\deg(\mu_n)\,<\,\cdots$.

$\bullet$ \ The type, ordinary or limit, of each augmentation $\mu_{n-1}<\mu_n$.  

$\bullet$ \ The valuations $\mu_n$ such that $\mu_n<\mu_{n+1}$ is an ordinary augmentation.\e

The conditions (MLV1), (MLV2) and (MLV4) are sufficient to derive these ``unicity" results. The condition (MLV3) is equivalent to 
\[
\mu_n(\phi_n)=\ga_n=\mu(\phi_n)\quad\mbox{for all }\ n.
\] This implies that all nodes of the chain, except for (eventually) $\mu$, are residue-transcendental.\e

The goal of this section is to show that a MLV chain for $\mu$ can be read directly in the descending family $\bopt=\bopt(\omu)$ of $\mu$-optimal ultrametric balls, associated  to any extension $\omu$ of $\mu$ to $\kbx$ in Section  \ref{subsecCompleteseq}.   

We need a well-known auxiliary observation. For the ease of the reader we give a short proof of it.

\begin{Lema}\label{aux}
Let $\nu,\mu$ be valuations on $\kx$ such that $\nu<\mu$.
Let $\phi\in\kpm$ be a key polynomial for $\mu$ of minimal degree. 
\begin{enumerate}
\item[(i)] If $\phi\in\ty(\nu,\mu)$, then $\mu=[\nu;\,\phi,\mu(\phi)]$.
\item[(ii)] Suppose that $\nu$ is the initial valuation of an increasing family $\cc=\left\{\rho_i\right\}_{i\in I}$ such that $\rho_i<\mu$ for all $i\in I$. If $\phi\in\kpi(\cc)$, then $\mu=[\cc;\,\phi,\mu(\phi)]$.
\end{enumerate}
\end{Lema}

\begin{proof}
In case (i), we may consider the ordinary augmentation $\mu'=[\nu;\,\phi,\mu(\phi)]$, because $\phi\in\ty(\nu,\mu)\sub\kp(\nu)$ and $\nu(\phi)<\mu(\phi)$.
In case (ii), we may consider the limit augmentation $\mu'=[\cc;\,\phi,\mu(\phi)]$, because $\phi\in\kpi(\cc)$ and $\rho_i(\phi)<\mu(\phi)$ for all $i\in I$.
In both cases, $\mu'=\mu$ because both valuations coincide on $\phi$-expansions.
\end{proof}

Let us recall the definition of $\bopt$.
If $\mu$ is valuation-transcendental and  $\omu=\vad$, then the balls  $B\in\bopt$ are $\mu$-optimal and  satisfy $B\supseteq \bad$. On the other hand, if $\mu$ is valuation-algebraic and $\omu=\lim_{i\in I}v_{a_i,\dta_i}$, then  the balls  $B\in\bopt$ are $\mu$-optimal and  satisfy $B\supseteq B(a_i,\dta_i)$ for some $i\in I$.\e

\noindent{\bf Notation. }\emph{For each $B=B(b,\ep)\in\bopt$, denote $\mu_B=\resk(v_{b,\ep})$.
}\e

Consider the sequence of degrees over $K$ of all balls in $\bopt$:
\begin{equation}\label{degrees}
1=m_0\,<\,m_1\,<\,\cdots\,<\,m_n\,<\,\cdots  
 \end{equation}
Denote by $\bopt_{m_n}\sub\bopt$ the subset of balls of degree $m_n$ over $K$. 

 \begin{Teo}\label{mainmlvchain}
 For each $n\in\N_0$ with $\bopt_{m_n}\ne\emptyset$, take
 a well-ordered cofinal family  $\bb_{m_n}\sub\bopt_{m_n}$ with respect to  descending inclusion, such that: 
\[
\bb_{m_n}=\{B\}\quad\mbox{whenever $\,\bopt_{m_n}$ has a last ball } B.
\]
Take  $\mu_n:=\mu_{B_n}$, for each $B_n:=\min\left(\bb_{m_n}\right)$. Then,  the chain
\[
\mu_0\,<\,\mu_1\,<\,\cdots\,<\,\mu_n\,<\,\cdots
\]	
is a Mac Lane-Vaqui\'e chain of $\mu$. 
\end{Teo}

\begin{proof}
By Theorem \ref{abskeypol} and Lemma-Definition \ref{optimal}, for each ball $B_n$ there exists $c\in\op{Min}_KB_n$ such that $\phi_n:=p_K(c)$ is an abstract key polynomial for $\mu$ satisfying $\mu_n=\mu_{\phi_n}$. By Lemma \ref{maximal}, $\phi_n$ is a key polynomial of minimal degree of $\mu_n$, so that $m_n=\deg_K B_n=\deg \phi_n=\deg(\mu_n)$. Thus, the chain satisfies (MLV1) because the sequence of degrees of the valuations $\mu_n$
coincides with the sequence (\ref{degrees}).\e

Consider any $\mu_{n-1}<\mu_n$ such that $\bb_{m_{n-1}}=\{B\}$. Since $B\supsetneq B_n$, Theorem \ref{mainthm} and Lemma \ref{kpim} show that $\phi_n\in\ty(\mu_{n-1},\mu_n)$. By Lemma \ref{aux}, $\mu_n$ is the ordinary augmentation $[\mu_{n-1};\,\phi_n,\ga_n]$, where $\ga_n=\mu_{\phi_n}(\phi_n)=\mu(\phi_n)$.

Consider any  $\mu_{n-1}<\mu_n$ such that $\bb_{m_{n-1}}=\left\{B_i\right\}_{i\in I_{n-1}}$ is infinite. Then, $\cc_{n-1}:=\left\{\mu_{B_i}\right\}_{i\in I_{n-1}}$ is  an increasing family of valuations of constant degree, with $\mu_{n-1}=\min\left(\cc_{n-1}\right)$.
Since $B_i\supsetneq B_n$ for all $i\in I_{n-1}$, Theorem \ref{mainlimitkp} shows that
 $\phi_n\in\kpi(\cc_{n-1})$. By Lemma \ref{aux}, $\mu_n=[\cc_{n-1};\,\phi_n,\ga_n]$, for $\ga_n=\mu_{\phi_n}(\phi_n)=\mu(\phi_n)$.

This proves that the chain satisfies (MLV2) and (MLV3).\e

 If $\mu$ is valuation-transcendental and $\omu=\vad$, then the last ball $\bad$ of $\,\bb(\omu)$ is $\mu$-optimal (Corollary \ref{lastBall}). Hence, if $\deg(\mu)=m_r$, we have $\bb_{m_r}=\{\bad\}$ and $\mu_r=\mu_{\phi_r}=\resk(\vad)=\mu$, by Lemma-Definition \ref{optimal}.  
 
Finally, suppose that  $\mu_{n-1}<\mu$ and
 $\bb_{m_{n-1}}=\left\{B_i\right\}_{i\in I}$ satisfies $\bigcap_{i\in I} B_i=\emptyset$. Let $\cc_{n-1}:=\left\{\mu_{B_i}\right\}_{i\in I}$. By Proposition \ref{reprove}, $\cc_{n-1}$ has a stable limit. 
Since $\nu:=\lim\left(\,\cc_{n-1}\right)$ is valuation-algebraic and $\nu\le\mu$, we have $\nu=\mu$ by Theorem \ref{EKP}.
\end{proof}

\noindent{\bf Remark. }In the Henselian case, Vaqui\'e showed as well, how to derive an MLV chain of a valuation-transcendental $\mu=\resk(v_{a,\dta})$, directly from the balls containing $B(a,\dta)$   \cite[Section 4]{Vaq3}.


\begin{thebibliography}{}

\bibitem{AFFGNR}M. Alberich-Carrami$\tilde{\mbox{n}}$ana, Alberto F. Boix, J. Fern\'andez, J. Gu\`ardia, E. Nart and J. Ro\'e, \emph{Of limit key polynomials}, Illin. J. Math. {\bf 65}, No.1 (2021), 201--229.

\bibitem{VT}M. Alberich-Carrami$\tilde{\mbox{n}}$ana,  J. Gu\`ardia, E. Nart and J. Ro\'e, \emph{Valuative trees of valued fields}, J. Algebra {\bf 614} (2023), 71--114.


\bibitem{AlePop} V. Alexandru and N. Popescu, \emph{Sur une classe de prolongements \`a K(X) d'une valuation sur un corps $K$}, Rev. Roumaine Math. Pures Appl. \textbf{33}  (1988), no. 5, 393-400.

\bibitem{APZTheorem} V. Alexandru, N. Popescu and A. Zaharescu, \emph{A theorem of characterization of residual transcendental extensions of a valuation}, J. Math. Kyoto Univ. {\bf 28} (1988), no. 4, 579-592.

\bibitem{APZ} V. Alexandru, N. Popescu and A. Zaharescu, \emph{Minimal pairs of definition of a residual transcendental extension of a valuation}, J. Math. Kyoto Univ. {\bf 30} (1990), no. 2, 207--225.

\bibitem{APZ2} V. Alexandru, N. Popescu and A. Zaharescu, \emph{All valuations on $K(X)$}, J. Math. Kyoto Univ. 30 (1990), no. 2, 281-296.




\bibitem{Andrei} A. Bengus-Lasnier, \textit{Minimal Pairs, Truncation and Diskoids}, J. Algebra \textbf{579} (2021), 388--427.

\bibitem{Dec}J. Decaup, W. Mahboub, M. Spivakovsky, \emph{Abstract key polynomials and comparison theorems with the key polynomials of Mac Lane-Vaqui\'e}, Illin. J. Math. {\bf 62} (2018), Number 1-4, 253--270.




















\bibitem{hos}F.J. Herrera Govantes, M.A. Olalla Acosta and M. Spivakovsky, \emph{Valuations in algebraic field extensions}, J. Algebra {\bf 312} (2007), no. 2, 1033--1074.

\bibitem{hos1}F.J. Herrera Govantes, W. Mahboub, M.A. Olalla Acosta  and M. Spivakovsky, \emph{Key polynomials for simple extensions of valued fields}, J. Singul. {\bf 25} (2022), 197--267.

\bibitem{Kap} I Kaplansky, \emph{Maximal fields with valuations}, Duke Math. J. \textbf{9(2)} (1942),  303--321.



\bibitem{Kuhl}F.-V. Kuhlmann, \emph{Value groups, residue fields, and bad places of rational function fields}, Trans. Amer. Math. Soc. {\bf 356} (2004), no. 11, 4559--4660.



\bibitem{KuhlAT}F.-V. Kuhlmann, \emph{Approximation types describing extensions of valuations to rational function fields}, Math. Z. {\bf 301} (2022), 2509--2546. 


\bibitem{mcla} S. Mac Lane, \emph{A construction for absolute values in polynomial rings}, Trans. Amer. Math. Soc. {\bf40} (1936), 363--395.








\bibitem{KP} E. Nart, \emph{Key polynomials over valued fields}, Publ. Mat. {\bf 64} (2020), 195--232.

\bibitem{MLV} E. Nart, \emph{Mac Lane-Vaqui\'e chains of valuations on a polynomial ring}, Pacific J. Math. {\bf 311} (2021), no. 1, 165--195.

\bibitem{Rig} E. Nart, \emph{Rigidity of valuative trees under henselization}, Pacific J. Math. {\bf 319} (2022), 189--211.

\bibitem{NN} E. Nart and J. Novacoski,  \emph{The defect formula}, Adv. Math. {\bf 428} (2023), 109153.

\bibitem{NartJosnei} E. Nart and J. Novacoski, \emph{Geometric parametrization of valuations on a polynomial ring in one variable}, Math. Z. \textbf{304} (2023), article number 59. 

\bibitem{NNMLKP} E. Nart and J. Novacoski, \emph{Minimal limit key polynomials}, J. Lond. Math. Soc. \textbf{111} issue 5 (2025), e70162. 



\bibitem{N2019} J. Novacoski, \emph{Key polynomials and minimal pairs}, J. Algebra {\bf 523} (2019), 1--14.



\bibitem{NS2018} J. Novacoski and M. Spivakovsky, \emph{Key polynomials and pseudo-convergent sequences}, J. Algebra {\bf 495} (2018), 199--219.






\bibitem{CaioJosnei} J. Novacoski and C. H. Silva de Souza, \emph{Parametrizations of subsets of the space of valuations}, Math. Z. {\bf 307} (2024),  article number 72.




\bibitem{Ostr} A. Ostrowski, \emph{Untersuchungen zur arithmetischen Theorie der K\"orper}, Math. Z. \textbf{39} (1935), 269--404.

\bibitem{PS} G. Peruginelli and D. Spirito. \emph{The Zariski-Riemann space of valuation domains associated to pseudo-convergent sequences}, Trans. Amer. Math. Soc. {\bf 373} (2020), no. 11, 7959--7990.

\bibitem{PS2} G. Peruginelli and D. Spirito, \emph{Extending valuations to the field of rational functions using pseudo-monotone sequences}, J. Algebra {\bf 586} (2021) 756--786.


\bibitem{PP} L. Popescu and N. Popescu, \emph{On the residual transcendental extensions of a valuation. Key polynomials and augmented valuations}, Tsukuba Journal of Mathematics  {\bf 15} (1991), 57--78.









\bibitem{Vaq}M. Vaqui\'e, \emph{Extension d'une valuation}, Trans. Amer. Math. Soc.  {\bf 359} (2007), no. 7, 3439--3481.




\bibitem{Vaq3}M. Vaqui\'e, \emph{Valuation augment\'ee, paire minimal et valuation approch\'ee}, preprint 2021, hal-02565309, version 2.


\end{thebibliography}
\end{document}